\documentclass[12pt]{article}

\usepackage{amsmath, amsthm, amssymb}  
\usepackage{graphicx}                  
\usepackage{geometry}                  
\usepackage{enumerate}
\usepackage{booktabs}                 
\usepackage{cite}                     
\usepackage{mathrsfs}                 
\usepackage{float}                    
\usepackage[ruled,vlined,linesnumbered]{algorithm2e}
\SetKwInput{KwIn}{Input}
\SetKwInput{KwOut}{Output}
\SetKwComment{tcp}{\% }{}

\usepackage{subcaption} 
\usepackage{enumitem}
\usepackage[colorlinks=true,linkcolor=blue,citecolor=magenta]{hyperref}
\usepackage{booktabs, tabularx}
\usepackage{tcolorbox}
\usepackage{authblk}

\usepackage{optidef}   

\usepackage{pdflscape} 
\usepackage{rotating}     
\usepackage{caption}
\title{\textbf{A Generalized Heron--Waist Problem: Optimality Conditions and Convergence Analysis}}

\author{
	Manohar Choudhary\thanks{
		Department of Mathematics and Statistics,
		Dr.~Hari Singh Gour Vishwavidyalaya, Sagar, India.
		E-mail: \texttt{manoharfbg@gmail.com}
	}
	\;\;and\;\;
	Triloki Nath\thanks{
		Department of Mathematics and Statistics,
		Deen Dayal Upadhyaya Gorakhpur University, Gorakhpur, India.
		E-mail: \texttt{tnverma07@gmail.com}
	}
	\;\;and\;\;
	Ram K.~Pandey\thanks{
		Department of Mathematics and Statistics,
		Dr.~Hari Singh Gour Vishwavidyalaya, Sagar, India.
		E-mail: \texttt{pandeywavelet@gmail.com}
	}
}

\date{}

\theoremstyle{plain}
\newtheorem{theorem}{Theorem}[section]

\newtheorem{proposition}[theorem]{Proposition}
\newtheorem{corollary}[theorem]{Corollary}

\theoremstyle{definition}
\newtheorem{definition}[theorem]{Definition}
\newtheorem{remark}[theorem]{Remark}
\newtheorem{example}[theorem]{Example}
\DeclareMathOperator{\proj}{proj}

\newtheorem{assumption}{Assumption}

\usepackage{rotating}  
\usepackage{booktabs}  

\begin{document}
	
	\maketitle

\begin{abstract}
	This paper introduces and solves the Generalized Heron–Waist Problem (GHWP),  that integrates the classical Heron problem of optimal hub location and the waist problem of minimal-perimeter configuration. The GHWP seeks an optimal closed polygonal chain with weights whose vertices are constrained to lie in the given nonempty, closed, and convex sets, while simultaneously minimizing weighted distances to a central hub point. This coupled formulation naturally models systems in which cyclic internal connectivity and radial access to a hub must be optimized jointly—a structural feature that arises in applications such as supply-chain design, transportation planning, and communication infrastructures. Using modern convex analysis tools, we establish existence of optimal solutions under boundedness and general position assumptions of sets, we prove uniqueness when constraint sets are strictly convex with positive weights. We also derive first order necessary and sufficient optimality conditions using subdifferential calculus. For computation, we develop a Projected Subgradient Algorithm (PSA) and we prove convergence of the best-iterate sequence under classical diminishing step size rules. Numerical illustrations in $\mathbb{R}^2$ and $\mathbb{R}^3$ are provided to validate the algorithm's robustness across diverse geometries and weighting schemes.
\end{abstract}

	\noindent\textbf{Keywords:} Convex Optimization, Heron Problem, Waist Problem, Facility Location, Subdifferential Calculus, Projected Subgradient Algorithm, Geometric Optimization.
	
	
\section{Introduction}

Geometric optimization problems, which seek to minimize distances, perimeters, or related quantities subject to spatial constraints, have fascinated mathematicians for centuries. One of the earliest and most celebrated distance-minimization problems is the classical \emph{Heron problem}, attributed to Heron of Alexandria \cite[p. 304]{GiMo12}.

In its classical form, the problem asks the following: given two points $P$ and $Q$ located on the same side of a straight line $\ell$ in the plane, find a point $M$ on $\ell$ that minimizes the sum of distances $|M - P| + |M - Q|$. Its elegant solution is governed by a geometric reflection principle, which reveals that the optimal point $M$ satisfies an \emph{equal-angle condition}: the line segments $\overline{MP}$ and $\overline{MQ}$ form equal angles with $\ell$, mirroring the law of reflection in optics.

This foundational problem inspired numerous generalizations over subsequent centuries. In the seventeenth century, Pierre de Fermat posed what is now known as the \emph{Fermat--Torricelli problem}: find a point in the plane that minimizes the sum of distances to three given points. Evangelista Torricelli provided its solution, establishing that when all angles of the triangle formed by the three points are less than $120^\circ$, the minimizer lies in the interior of the triangle and forms $120^\circ$ angles with each pair of connecting segments \cite{Spindler2025}. Later, Endre Weiszfeld \cite{WePl09,We37} extended this framework to $n$ points by introducing an iterative algorithm that bears his name and remains influential in modern facility location theory.

\subsection{Modern Generalizations via Convex Analysis}

The development of convex analysis in the mid-twentieth century, pioneered by Rockafellar \cite{Ro70}, Moreau, and others, provided powerful tools for revisiting classical geometric problems in a broadly applicable convex analysis framework.
	
Motivated by applications in location science and optimization, Mordukhovich, Nam, and Salinas \cite{Mordukhovich2012} introduced a convex-analysis generalization of the classical Heron problem by replacing points and lines with nonempty closed convex sets. In this framework, they formulated and solved the \emph{generalized Heron problem (GHP)}:
\begin{equation}\label{GHP}
	\min_{x \in S} D(x) := \sum_{i=1}^{n} d(x; C_i),
\end{equation}
where $S$ and $C_i$, $i = 1, \ldots, n$, are nonempty closed convex subsets of $\mathbb{R}^s$, and
\[
d(x; C_i) := \inf \{ \|x - y\| : y \in C_i \}
\]
denotes the Euclidean distance from the point $x$ to the set $C_i$.

\begin{figure}[H]
	\centering
	\includegraphics[scale=0.15]{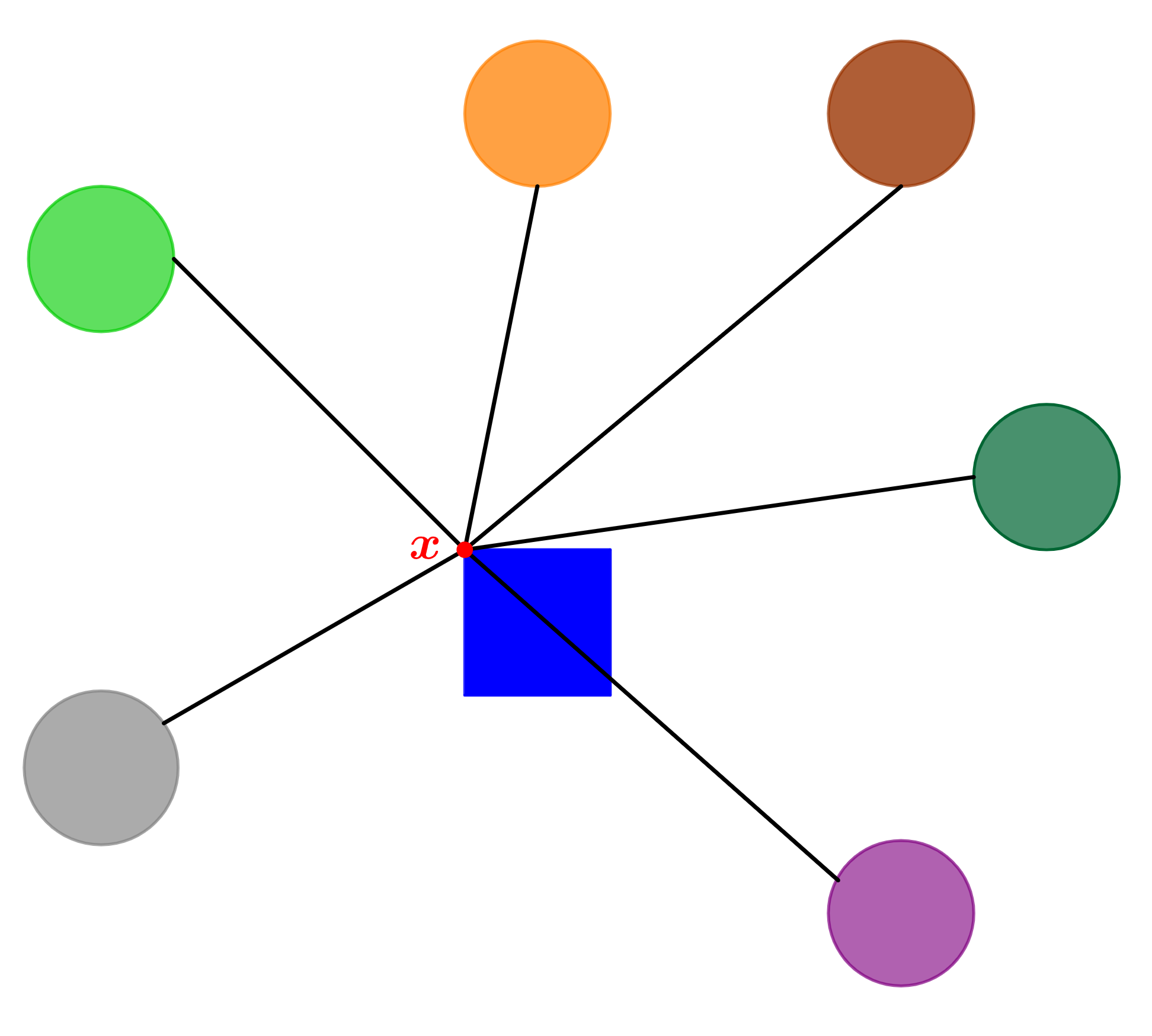}
	\caption{Illustration of the Generalized Heron Problem (GHP).}
	\label{fig:GHP_Illustration}
\end{figure}

This formulation replaces the classical constraint line with a closed convex feasible region $S$ and points  with convex target sets $C_i$, thereby substantially broadening the geometric and practical scope of the problem. Using tools from subdifferential calculus, the authors derived complete necessary and sufficient optimality conditions and developed convergent numerical algorithms based on projected subgradient methods. From an applied perspective, the generalized Heron problem models the search for an optimal hub location within a convex region that minimizes the total sum of distances to multiple convex sets. This development represents a fundamental transition of the Heron problem from classical geometric constructions to the modern theory of nonsmooth convex optimization. Recently, we further extended this problem to multiple hub sets \cite{Nath2025}. 
	
Parallel to the Heron problem, another classical geometric optimization model is the \emph{waist problem}, originally proposed by John Tyrrell \cite[p. 223]{BrTi11}. The waist problem seeks a triangle of minimal perimeter whose vertices lie on three given, pairwise disjoint lines in space. Geometrically, it can be interpreted as an elastic band wrapped around rigid supports, contracting to a configuration of minimal total length. From an optimization perspective, the waist problem may also be viewed as a \emph{facility-location--type problem of distributed nature}: rather than connecting facilities to a single hub, it optimizes the cyclic connectivity among the facilities themselves. Convexity plays a crucial role in this setting, ensuring the existence and uniqueness of solutions as well as equilibrium configurations. The optimal triangle satisfies reflection-type angle conditions that closely parallel those arising in the classical Heron problem, see \cite[p. 304]{GiMo12}.

Building on this classical foundation, We \cite{NaCh25Waist} recently introduced the \emph{generalized waist problem (GWP)}, which extends the waist problem by replacing the lines with closed convex sets. The resulting optimization model is formulated as
\begin{equation}\label{Waist}
	\min_{a \in C} D(a) := \sum_{i=1}^{m} \|a_i - a_{i+1}\|, 
	\qquad a_{m+1} := a_1, 
\end{equation}
where
\[
C := C_1 \times C_2 \times \cdots \times C_m,
\]
with each $C_i \subset \mathbb{R}^n$ being nonempty, closed, and convex. Here, the decision variable $a = (a_1, \ldots, a_m) \in C$ represents a closed polygonal chain whose vertices are constrained to lie in their respective convex sets. This generalization significantly broadens the geometric scope of the waist problem and places it naturally within the modern convex-analysis optimization framework.

\begin{figure}[H] 
	\centering	
	\includegraphics[scale=.12]{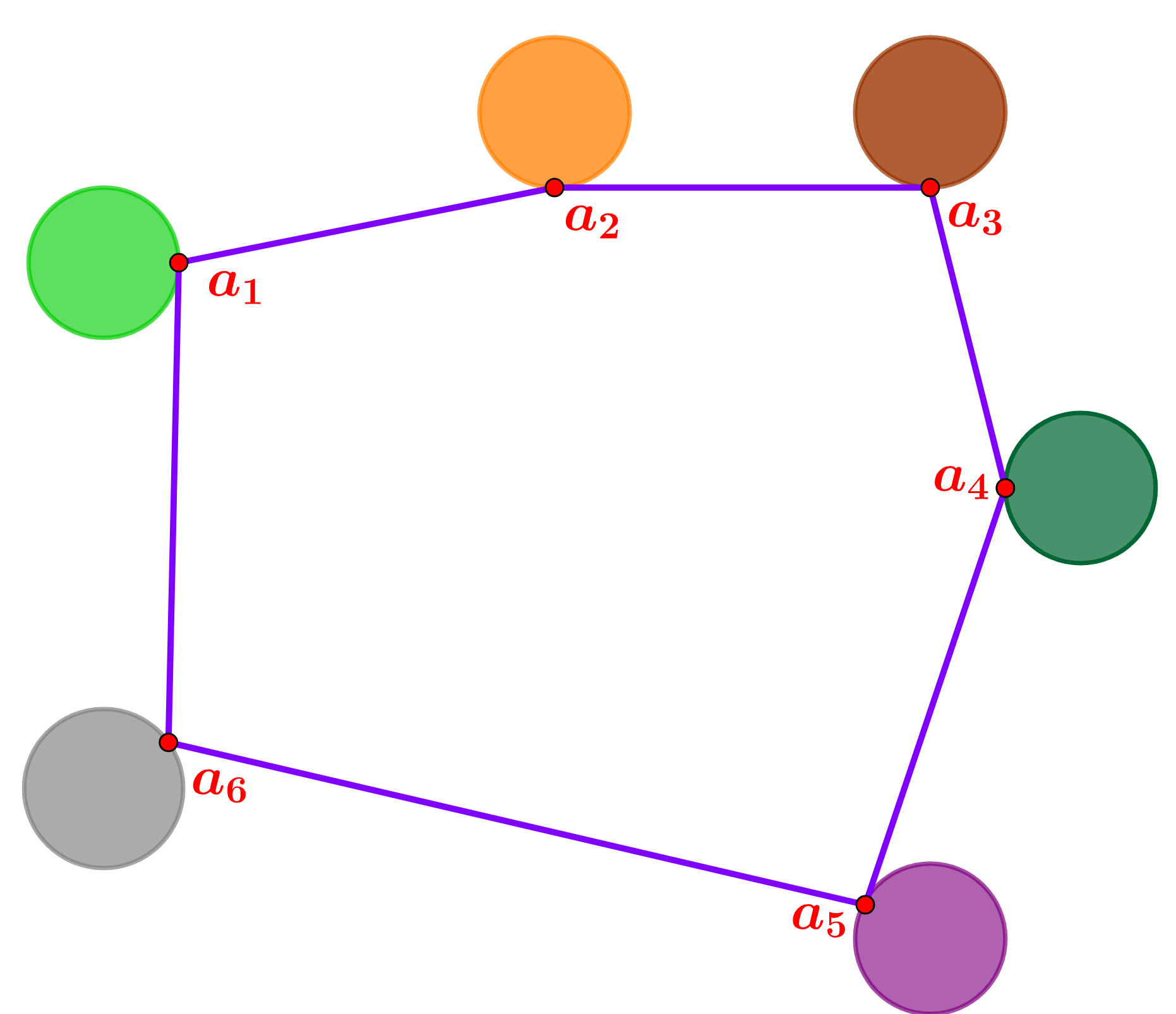} 
	\caption{Illustration of the generalized waist problem in $\mathbb{R}^2$.}
\end{figure}

Despite these advances, existing models continue to treat \emph{hub-based connectivity} (as in Heron-type problems) and \emph{cyclic connectivity} (as in waist-type problems) in isolation. In contrast, many real-world systems inherently require the \emph{simultaneous optimization of internal network connectivity and external access to a central facility}. Typical examples include supply-chain loops connected to distribution hubs, transportation cycles linked to depots, sensor networks coordinated by a control unit, and communication backbones accessing centralized servers. Such systems cannot be adequately captured by either Heron-type or waist-type formulations alone.

Motivated by this observation, we propose a \emph{hybrid geometric optimization framework} that integrates both paradigms within a single model. Moreover, practical applications demand \emph{heterogeneity}: different connections may lead to different costs, priorities, or levels of importance. To reflect this reality, we introduce heterogeneous positive weights that allow each interaction---whether between neighboring facilities or between facilities and a central hub---to be independently scaled, emphasized, or effectively neutralized.

In this paper, we introduce and solve the \emph{Generalized Heron--Waist Problem (GHWP)}, a unified convex optimization model that simultaneously captures hub-based and cyclic connectivity objectives under general convex constraints. In Section~\ref{Formulation_GHWP}, we formulate the GHWP and discuss its fundamental structural properties. Section~\ref{characterization} is devoted to the theoretical analysis of the problem, where we establish existence and uniqueness of optimal solutions and derive complete necessary and sufficient optimality conditions with clear geometric force-balance interpretations. In Section~\ref{algorithms}, we develop a projected subgradient algorithm tailored to the nonsmooth convex structure of the GHWP and prove its convergence under standard diminishing step-size rules. Section~\ref{numerical} presents detailed numerical experiments in both $\mathbb{R}^2$ and $\mathbb{R}^3$, illustrating the practical performance, robustness, and scalability of the proposed method. Finally, Section~\ref{conclusion} concludes the paper with a summary of the main contributions and directions for future research.

\subsection{Problem Formulation} \label{Formulation_GHWP}
	We work throughout in the $n$-dimensional Euclidean space $\mathbb{R}^n$, equipped with the standard inner product $\langle \cdot, \cdot \rangle$ and the induced Euclidean norm $\|\cdot\|$. A set $C \subset \mathbb{R}^n$ is said to be \emph{convex} if for any $x, y \in C$ and any $\lambda \in [0,1]$, the convex combination $\lambda x + (1-\lambda)y$ belongs to $C$. It is \emph{strictly convex} if, for any two distinct points $x, y \in C$ and any $\lambda \in (0,1)$, the point $\lambda x + (1-\lambda)y$ lies in the interior of $C$, denoted by $\operatorname{int}(C)$. Throughout this paper, all constraint sets are assumed to be nonempty, closed, and convex.
	The Generalized Heron--Waist Problem (GHWP) involves two distinct classes of decision variables: a cyclic chain variable and a central hub variable. 
	\begin{itemize}
		\item \textbf{Chain Variables:}  The configuration of the polygonal chain is described by the vector $\mathbf{a} = (a_1, a_2, \ldots, a_m) \in (\mathbb{R}^n)^m, $
		where each variable $a_i \in \mathbb{R}^n$ is constrained to lie in a prescribed set $C_i \subset \mathbb{R}^n$. Geometrically, the vector $\mathbf{a}$ represents the variable of a closed polygonal chain: line segments connect $a_1$ to $a_2$, $a_2$ to $a_3$, and so on, with the path closing via the segment joining $a_m$ to $a_1$. To simplify notation, we adopt the cyclic index convention 	$a_{m+1} := a_1,$ 	so that the final edge length is expressed as $\|a_m - a_{m+1}\| = \|a_m - a_1\|$. 
		
		\item  \textbf{Hub Variable:} In addition to the chain, the model incorporates a central hub variable whose location is given by $x \in \mathcal{S} \subset \mathbb{R}^n$. It acts as a central facility from which radial connections extend to all vertices of the chain.
	\end{itemize}
	
For notational convenience, we collect all decision variables into a single composite vector $\mathbf{u} := (a_1, a_2, \ldots, a_m, x) \in \mathbb{R}^{n(m+1)},$ where $a_i \in \mathbb{R}^n$, $i=1,\ldots,m$, denote the vertices of a closed polygonal chain and $x \in \mathbb{R}^n$ represents a central hub location.
The feasible set for the composite variable $\mathbf{u}$ is defined as the Cartesian product
\begin{equation}
	\Omega := C_1 \times C_2 \times \cdots \times C_m \times \mathcal{S}
	\subset \mathbb{R}^{n(m+1)},
	\label{eq:feasible_region}
\end{equation}
where the sets $C_1, C_2, \ldots, C_m, \mathcal{S} \subset \mathbb{R}^n$ are assumed to be
nonempty, closed, and convex.
Consequently, the feasible region $\Omega$ is itself nonempty, closed, and convex.

\vspace{0.3cm}

Let $\boldsymbol{\rho} = (\rho_1, \ldots, \rho_m)$ and
$\boldsymbol{\omega} = (\omega_1, \ldots, \omega_m)$ be vectors of \emph{nonnegative weights},
where $\rho_i \ge 0$ and $\omega_i \ge 0$ for all $i = 1, \ldots, m$.
The weights $\rho_i$ quantify the cost or importance associated with the connectivity between
successive variables of the polygonal chain, while the weights $\omega_i$ represent the cost or
priority of connecting the chain variables to the central hub.

We define the objective function $J : \Omega \to \mathbb{R}$ by
\begin{equation}\label{eq:objective_ghwp}
	J(\mathbf{u}) = J(a,x)
	:=
	\underbrace{\sum_{i=1}^{m} \rho_i \|a_i - a_{i+1}\|}_{\text{Waist term } W(a)}
	+
	\underbrace{\sum_{i=1}^{m} \omega_i \|a_i - x\|}_{\text{Heron term } H(a,x)},
\end{equation}
where the cyclic convention $a_{m+1} := a_1$ is adopted.

These considerations naturally lead to the following optimization problem, which constitutes the central object of study in this paper.

\begin{equation}\label{Heron_Waist_Problem}
	\min_{\mathbf{u} \in \Omega} J(\mathbf{u}),
\end{equation}
which we refer to as the \emph{Generalized Heron--Waist Problem with nonnegative weights}.

The first term $W(a)$ measures the weighted total length of the closed polygonal chain,
generalizing the classical waist (perimeter) minimization problem.
The second term $H(a,x)$ represents the weighted total distance from the chain vertices to the hub,
generalizing the classical Heron-type facility location problem.
Allowing nonnegative weights provides modeling flexibility: setting certain weights to zero
effectively neutralizes the corresponding geometric interactions, while positive weights reflect
different costs or priorities in practical applications.

\begin{figure}[H]
	\centering
\includegraphics[scale=0.45]{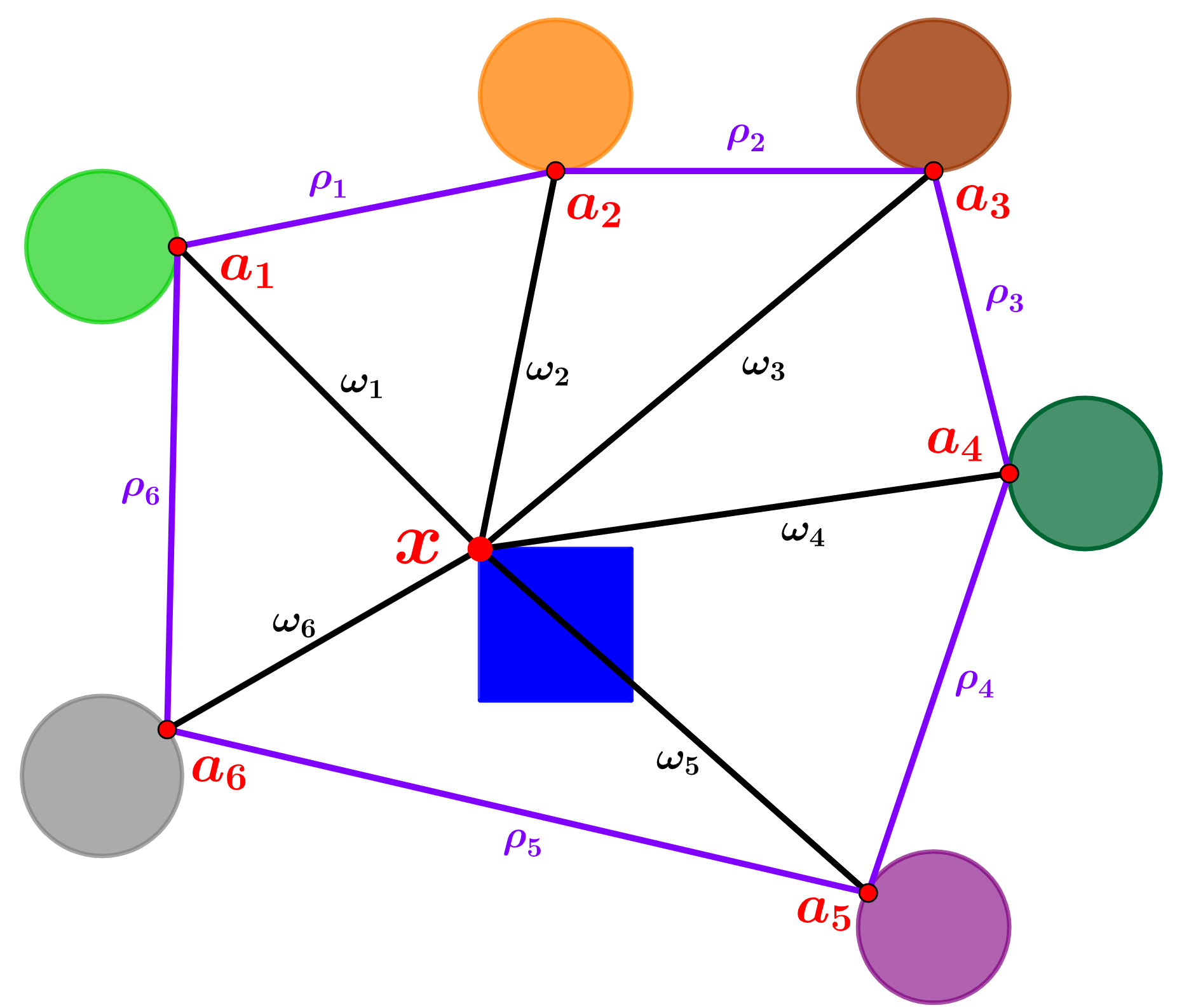}
	\caption{Illustration of the Generalized Heron-Waist Problem (GHWP).}
	\label{fig:GHP_Illustration}
\end{figure}

A critical assumption is imposed to ensure that each chain variable actively
participates in the optimization process.

\begin{assumption}
	\label{assump:nondegeneracy}
	For each $k \in \{1, 2, \ldots, m\}$, $	\rho_{k-1} + \rho_k + \omega_k > 0$ where the indices are interpreted cyclically, that is, $\rho_0 := \rho_m$.
\end{assumption}

\begin{remark}[Necessity of the Nondegeneracy Condition]
	\label{rem:nondegeneracy_necessity}
	If $\rho_{k-1} + \rho_k + \omega_k = 0$ for some index $k$, then the variable
	$a_k$ appears in the objective function only through the terms
	\[
	\rho_{k-1} \|a_{k-1} - a_k\|, \quad
	\rho_k \|a_k - a_{k+1}\|, \quad
	\omega_k \|a_k - x\|,
	\]
	all of which vanish. Consequently:
	\begin{enumerate}
		\item The objective functional $J(\mathbf{u})$ becomes \emph{independent} of
		$a_k \in C_k$.
		\item The constraint set $C_k$ is effectively \emph{inactive}, since any choice
		of $a_k \in C_k$ yields the same objective value.
		\item The variable $a_k$ is completely \emph{decoupled} from the optimization
		problem, leading to artificial non-uniqueness of solutions.
		\item The first-order optimality conditions become degenerate, as no balancing
		forces act on $a_k$ to determine its location.
	\end{enumerate}
\end{remark}

Assumption~\ref{assump:nondegeneracy} guarantees that every chain variable is involved in at least one active interaction—either with a neighboring variable or with the hub variable $x$—thereby ensuring a fully coupled and meaningful optimization problem.

To further clarify the role of Assumption~\ref{assump:nondegeneracy}, Figure~\ref{fig:nondegeneracy_illustration} provides a geometric illustration of a fully coupled configuration in which every chain variable actively participates in the optimization process. The figure highlights how each chain point interacts either with its neighboring variables through the waist terms or with the hub through the radial terms, ensuring that all variables are subject to balancing forces. This visual representation reinforces the necessity of the nondegeneracy condition for maintaining meaningful coupling and well-defined optimality conditions

\begin{figure}[H]
	\centering
	\includegraphics[width=0.50\textwidth]{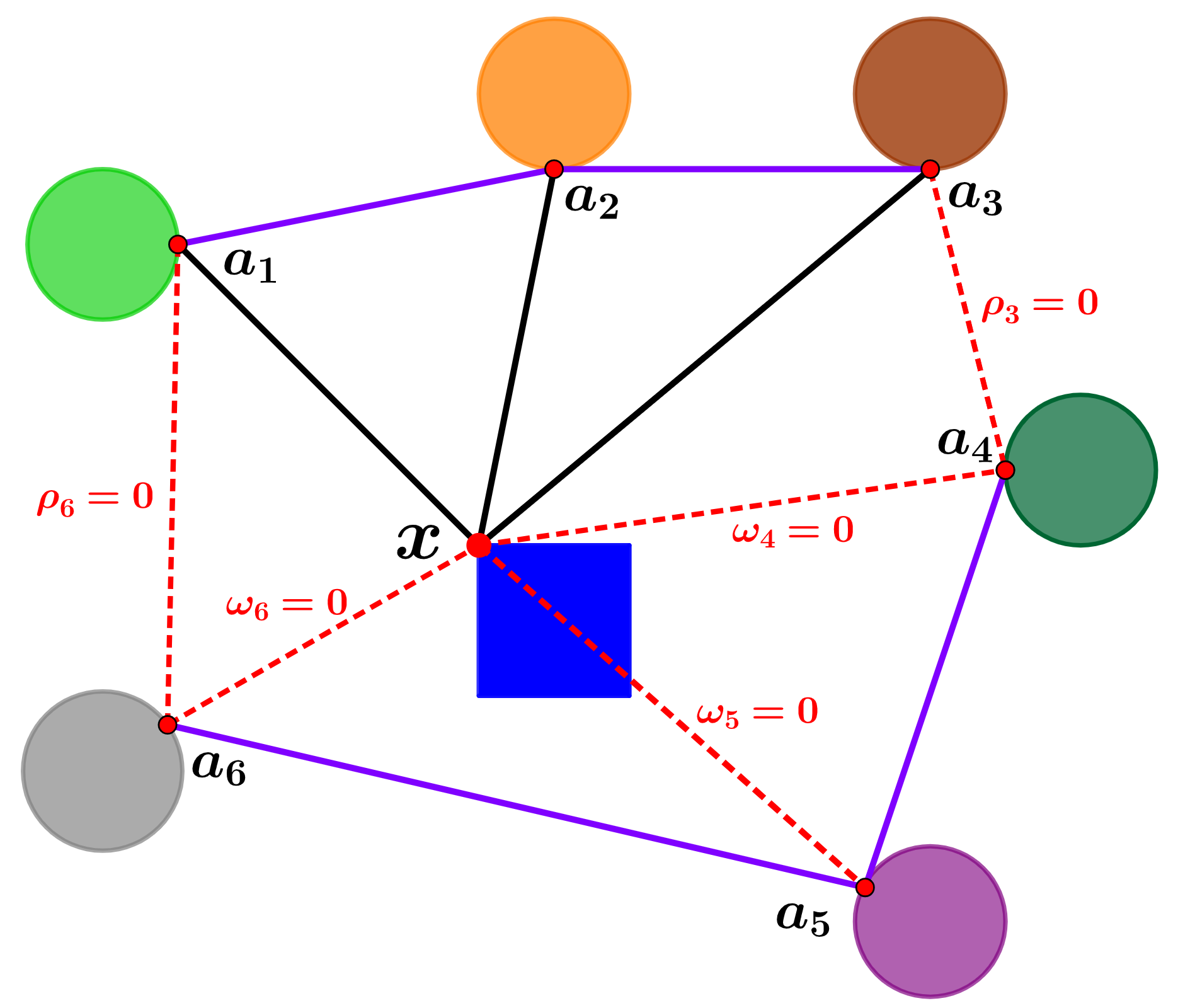}
	\caption{Illustration of the nondegeneracy condition in the Generalized Heron--Waist Problem.}
	\label{fig:nondegeneracy_illustration}
\end{figure}

\section{Preliminaries}
\label{sec:preliminaries}

In this section, we recall essential concepts and results from convex analysis that support both the theoretical characterization of optimal solutions for GHWP and the development of computational algorithms. A fundamental advantage of convex optimization is that every local minimum is a global minimum, and this property extends naturally to our problem due to the convex structure of both the objective function and the feasible region. We refer the reader to standard texts in convex analysis (e.g., \cite{Ro70, HiLe96, MoNa23}) for comprehensive background.

We begin with the fundamental notion of convexity for functions, which plays a pivotal role throughout our analysis.

\begin{definition}
	\label{def:effective_domain}
	For a function $f : \mathbb{R}^n \to (-\infty, +\infty]$, the \emph{effective domain} is the set
	\[
	\mathrm{dom}\, f := \{ x \in \mathbb{R}^n \mid f(x) < +\infty \}.
	\]
\end{definition}

\begin{definition}
	\label{def:convex_function}
	A function $f : \mathbb{R}^n \to (-\infty, +\infty]$ is \emph{convex} if $\mathrm{dom}\, f$ is a convex set and
	\[
	f(\lambda x + (1 - \lambda) y) \le \lambda f(x) + (1 - \lambda) f(y)
	\]
	for all $x, y \in \mathrm{dom}\, f$ and $\lambda \in [0,1]$. The function is \emph{strictly convex} if the inequality is strict whenever $x \neq y$ and $\lambda \in (0,1)$.
\end{definition}

A geometric characterization of convexity is provided by the epigraph, which lifts a function into a higher-dimensional set.

\begin{definition} 
	\label{def:epigraph}
	Let $f : \mathbb{R}^n \to \mathbb{R} \cup \{+\infty\}$ be an extended real-valued function.  
	The \emph{epigraph} of $f$ is defined as
	\[
	\operatorname{epi} f := \bigl\{ (x,t) \in \mathbb{R}^n \times \mathbb{R} \;\big|\; f(x) \le t \bigr\}.
	\]
	
	The function $f$ is said to be \emph{convex} if and only if its epigraph $\operatorname{epi} f$ is a convex subset of $\mathbb{R}^{n+1}$.                    	Moreover, the function $f$ is called \emph{closed} if its epigraph $\operatorname{epi} f$ is a closed subset of $\mathbb{R}^{n+1}$.
\end{definition}

\begin{definition}
	For a nonempty set $C \subset \mathbb{R}^n$, the \emph{distance function} to $C$ is defined by
	\[
	d_C(x) := \inf_{y \in C} \|x - y\|.
	\]
	If $C$ is convex, then $d_C$ is a convex and Lipschitz continuous function on $\mathbb{R}^n$.
\end{definition}

\begin{definition}
	Let $C \subset \mathbb{R}^n$ be nonempty. The \emph{projection} of a point $x \in \mathbb{R}^n$ onto $C$ is defined as
	\[
	\proj_C(x) := \arg\min_{y \in C} \|x - y\|.
	\]
	If $C$ is closed and convex, the projection is uniquely defined. Moreover, the distance function satisfies
	\[
	d_C(x) = \|x - \proj_C(x)\|.
	\]
\end{definition}

Auxiliary extended real-valued functions play an important role in convex optimization, particularly in transforming constrained problems into unconstrained ones. One of the most fundamental such constructions is the indicator function.

\begin{definition}
	For a set $C \subset \mathbb{R}^n$, the \emph{indicator function} of $C$ is defined as
	\[
	\delta_C(x) :=
	\begin{cases}
		0, & x \in C, \\
		+\infty, & x \notin C.
	\end{cases}
	\]
\end{definition}

To establish the uniqueness of the solution of problem \eqref{Heron_Waist_Problem}, we use two basic concepts from convex analysis: the convex hull and the separation theorem. The convex hull helps us describe the geometric structure of the feasible region, while the separation theorem ensures that distinct points can be separated by a hyperplane under appropriate conditions. Together with the strict convexity of at least one set, these tools lead to the conclusion that the problem admits a unique solution. This motivates us to recall the following definition.

\begin{definition} \label{def:convex_hull}	Given a set \( S \subseteq \mathbb{R}^n \), its \emph{convex hull}, denoted \(\mathrm{conv}(S)\), is the smallest convex set containing \(S\). Equivalently, \cite[Definition 6.11 ]{Be14}
	\[
	\mathrm{conv}(S)
	\;=\;
	\Bigl\{
	\sum_{i=1}^{m} \lambda_i x_i
	\;\Big\vert\;
	x_i \in S, \;
	\lambda_i \ge 0, \;
	\sum_{i=1}^{m} \lambda_i = 1, \;
	m \in \mathbb{N}
	\Bigr\}.
	\]
	In other words, any point in \(\mathrm{conv}(S)\) can be written as a convex combination of points in \(S\).
\end{definition}

The Separation of two convex sets by means of a hyperplane plays an important role in optimization and analysis. In particular, the alternative theorem arising in study of optimality conditions are consequences of separation theorem \cite[p.124]{HiLe96}. Here are these important notions.

\begin{theorem}
	 Let $C \subset \mathbb{R}^n$ be a nonempty closed convex set and let 
	 $y \notin C$. Then there exist a vector $a \in \mathbb{R}^n$, 
	 $a \neq 0$, and a scalar $\alpha \in \mathbb{R}$ such that
	 \[
	 \langle a, y \rangle > \alpha \geq \langle a, x \rangle 
	 \quad \text{for all } x \in C.
	 \]
\end{theorem}
The sets of the form $\{x \in \mathbb{R}^n \mid \langle a,x \rangle = \alpha \}$ and $\{x \in \mathbb{R}^n \mid \langle a,x \rangle \leq \alpha \}$ are called a \emph{hyperplane} and a \emph{closed halfspace}, respectively.

\begin{theorem}\label{seperationtheorem}
	Let $C, D \subset \mathbb{R}^n$ be nonempty convex sets such that 
$C \cap D = \emptyset$. Then there exist a vector 
$a \in \mathbb{R}^n$, $a \neq 0$, and a scalar 
$\alpha \in \mathbb{R}$ such that
\[
\langle a, x \rangle \leq \alpha \leq \langle a, y \rangle 
\quad \text{for all } x \in C \text{ and for all } y \in D.
\]
\end{theorem}

In this case, we say that the sets $C$ and $D$ are separated by the hyperplane $H = \{ x \in \mathbb{R}^n \mid \langle a, x \rangle = \alpha \}$. If, in addition, the separation is strict, i.e., there exist  $a \in \mathbb{R}^n \setminus \{0\}$ and $\alpha \in \mathbb{R}$ such that $\langle a, x \rangle < \alpha \quad \forall x \in C$ and $\langle a, y \rangle > \alpha \quad \forall y \in D,$ then $C$ and $D$ are said to be \textbf{strictly separable}.

\begin{definition}
	Let $C \subset \mathbb{R}^n$ be a closed and convex and let $x \in C$. The \emph{normal cone} to $C$ at $x$ is defined as
	\[
	N_C(x) := \{\, v \in \mathbb{R}^n \mid \langle v, y - x \rangle \le 0 \;\; \forall\, y \in C \,\},
	\]
	Geometrically, $N_C(x)$ consists of all vectors that form a nonacute angle with every feasible direction emanating from $x$ within $C$.
\end{definition}

The following proposition provides a useful formula for computing normal cones to Cartesian products of convex sets, which will be repeatedly used in our analysis.

\begin{proposition}\cite[Page~59, Proposition~2.39]{DhDu11}
	\label{Normalcone_Product}
	Let $C_i \subset \mathbb{R}^{n_i}$ be closed convex sets and let $a_i \in C_i$ for $i = 1,2$.
	Then,
	\[
	N_{C_1 \times C_2}\bigl((a_1, a_2)\bigr)
	=
	N_{C_1}(a_1) \times N_{C_2}(a_2).
	\]
\end{proposition}

Since the distance function is generally nondifferentiable, tools from nonsmooth analysis are required. The central concept is the subgradient, which extends the classical notion of a gradient to convex functions but possibly nondifferentiable.

\begin{definition}
	Let $f : \mathbb{R}^n \to (-\infty, +\infty]$ be convex. A vector $v \in \mathbb{R}^n$ is called a \emph{subgradient} of $f$ at $x \in \mathrm{dom}\,f$ if
	\[
	f(y) \ge f(x) + \langle v, y - x \rangle
	\quad \forall\, y \in \mathbb{R}^n.
	\]
	The set of all subgradients of $f$ at $x$ is called the \emph{subdifferential} of $f$ at $x$ and is denoted by $\partial f(x)$.
\end{definition}

For indicator function of convex sets, the subdifferential admits a particularly simple characterization.

\begin{equation}\label{eq:indicator-subgrad}
	\partial \delta_C(x) = N_C(x), \quad x \in C,
\end{equation}
as shown in~\cite[Example~3.7, p.~85]{MoNa23}.

The subdifferential of the distance function plays a crucial role in both the theoretical analysis and the design of numerical algorithms for the Generalized Heron--Waist Problem.

\begin{proposition}[Subdifferential of the distance function]
	\label{subgradientformula}
	Let $C \subset \mathbb{R}^n$ be nonempty, closed, and convex. Then, for any $x \in \mathbb{R}^n$,
	\[
	\partial d_C(x) =
	\begin{cases}
		\left\{ \dfrac{x - \proj_C(x)}{\|x - \proj_C(x)\|} \right\}, & \text{if } x \notin C, \\[10pt]
		N_C(x) \cap \mathbb{B}, & \text{if } x \in C,
	\end{cases}
	\]
	where $\mathbb{B} := \{ z \in \mathbb{R}^n \mid \|z\| \le 1 \}$ is the closed unit ball.
	A proof can be found in~\cite[Example~3.3, p.~259]{HiLe96}.
\end{proposition}

In the special case where $C = \{y\}$ is a singleton, the projection satisfies $\proj_C(x) = \{y\}$ and $d_C(x) = \|x - y\|$. Hence, for $x \neq y$,
\begin{equation}
	\label{subgradientformula_forsingleton}
	\partial d_C(x) = \left\{ \frac{x - y}{\|x - y\|} \right\}.
\end{equation}

The following affine composition rule is a fundamental tool in subdifferential calculus.

\begin{theorem}[Affine Composition Rule, {\cite[Theorem 4.2.1, p.263]{HiLe96}}]\label{affinetheorem}
	Let  $T: \mathbb{R}^n \to \mathbb{R}^m$ 
	be an affine mapping defined by $T(x) \;=\; T_0 \,x \;+\; c,$ where  	$T_0: \mathbb{R}^n \to \mathbb{R}^m$
	is a linear operator and  \(c \in \mathbb{R}^m\).  
	Let \(\psi: \mathbb{R}^m \to \mathbb{R}\) be a finite convex function. Then, for every \(x \in \mathbb{R}^n\),
	\[
	\partial\bigl(\psi\circ T\bigr)(x)
	\;=\;
	T_0^*\,\partial \psi\bigl(T(x)\bigr),
	\]
	where $T_0^*: \mathbb{R}^m \to \mathbb{R}^n$
	denotes the adjoint operator of \(T_0\). In other words, \(T_0^*\) is the unique linear map satisfying
	\[
	\langle T_0\,x,\; v \rangle
	\;=\;
	\langle x,\; T_0^*\,v \rangle
	\quad
	\forall\, x \in \mathbb{R}^n,\; v \in \mathbb{R}^m.
	\]
\end{theorem}

Following the discussion on \cite[p.\ 263]{HiLe96}, Let $f: \underbrace{\mathbb{R}^n \times \cdots \times \mathbb{R}^n}_{m\text{ times}}	\;\longrightarrow\; \mathbb{R}$
be a convex function in \(m\) blocks of variables, i.e.\ we write 
\((x_1, x_2, \dots, x_m)\) with each \(x_i \in \mathbb{R}^n\).  
Fix any \((x_2,\dots,x_m) \in \mathbb{R}^{n \times (m-1)}\), and define the affine map $ T: \mathbb{R}^n \to \mathbb{R}^{n \times m},\quad T(x_1) = (x_1, x_2, \dots, x_m)$ 
The linear part of \( T \) is \( T_0(x_1) = (x_1, 0, \dots, 0) \), and the corresponding adjoint is \( T_0^*(v_1, v_2, \dots, v_m) = v_1 \). Hence \( f \circ T \) is precisely the function \( x_1 \mapsto f(x_1, x_2, \dots, x_m) \).

\medskip

For each \(1 \le i \le m\), define the partial function
\[
f_i: \mathbb{R}^n \to \mathbb{R},
\quad
f_i(x_i) 
\;=\;
f\bigl(x_1,\dots,x_{i-1},\;x_i,\;x_{i+1},\dots,x_m\bigr),
\]
where the other \(x_j\) (with \(j \neq i\)) are held fixed. Applying Theorem~\ref{affinetheorem} to each such affine composition gives:
\[
\partial f_i(x_i)
\;=\;
\Bigl\{
v_i \in \mathbb{R}^n 
\;\Bigm|\;
\exists\,v_j \,(j\neq i)\text{ with }(v_1,\dots,v_m)\in \partial f(x_1,\dots,x_m)
\Bigr\}.
\]
Thus, \(\partial f_i(x_i)\) is the \emph{projection} of \(\partial f(x_1,\dots,x_m)\) onto the \(i\)th block. Consequently,
\begin{equation}\label{subset_inclusion}
	\partial f(x_1,\dots,x_m)
	\;\subseteq\;
	\partial f_1(x_1)
	\;\times\;
	\partial f_2(x_2)
	\;\times\;
	\cdots
	\;\times\;
	\partial f_m(x_m).
\end{equation}

\begin{remark}\label{remark:subgradient_equality}
	In general, the inclusion above in \eqref{subset_inclusion}  is strict. However, if for each \(i\) the partial function \(f_i\) is differentiable at \(x_i\), then \(\partial f_i(x_i)\) reduces to a singleton \(\{\nabla f_i(x_i)\}\). Consequently, if all \(f_i\) are differentiable at the respective \(x_i\) for \(i = 1, \dots, m\), the subdifferential becomes as:
	\begin{equation}\label{sub_cartessian}
		\partial f(x_1,\dots,x_m)
		\;=\;
		\Bigl\{
		\bigl(\nabla f_1(x_1),\,\nabla f_2(x_2),\,\dots,\,\nabla f_m(x_m)\bigr)
		\Bigr\}.
	\end{equation}
\end{remark}

\begin{example}[\textbf{Subdifferential of the pairwise distance}]
	\label{ex:subdiff_distance}
	Let \( D(a_1, \dots, a_m) = \|a_1 - a_2\| \) denote the pairwise distance between \( a_1 \) and \( a_2 \) in \( \mathbb{R}^n \), $a_1 \ne a_2$ viewed as a function of the tuple \( (a_1, \dots, a_m) \in \mathbb{R}^{n \times m}\). The subdifferential of \( D \) with respect to the full tuple is given by the ordered tuple of partial subdifferentials:
	\[
	\partial D(a_1, \dots, a_m) = \bigl( \partial D_1(a_1),\, \partial D_2(a_2),\, \partial D_3(a_3),\, \dots,\, \partial D_m(a_m) \bigr),
	\]
	where \( \partial D_k(a_k) \) denotes the projection of \( \partial D(a_1, \dots, a_m) \) onto the \( k \)-th coordinate.
	
	By Remark~\ref{remark:subgradient_equality} and the preceding discussion, we obtain:
	\[
	\partial D_1(a_1) = \frac{a_1 - a_2}{\|a_1 - a_2\|}, \qquad
	\partial D_2(a_2) = \frac{a_2 - a_1}{\|a_1 - a_2\|}, \qquad
	\partial D_k(a_k) = \mathbf{0} \quad \text{for all } k \ge 3.
	\]
	Applying the Cartesian product rule for subdifferentials see \eqref{sub_cartessian}, we conclude that
	\[
	\partial D(a_1, \dots, a_m) =
	\biggl(
	\frac{a_1 - a_2}{\|a_1 - a_2\|},\;
	\frac{a_2 - a_1}{\|a_1 - a_2\|},\;
	\mathbf{0}, \dots, \mathbf{0}
	\biggr).
	\]
	Thus, only the first two components contribute to the subdifferential while the remaining \( m - 2 \) components are zero.
\end{example} 

Following is the most fundamental result regarding subdifferentials calculus and is widely used as sum rule.

\begin{theorem}[Moreau--Rockafellar Theorem {\cite[Corollary~3.21, p.~93]{MoNa23}}]
	\label{sumofsubgradient}
	Let \(\psi_i : \mathbb{R}^n \to (-\infty, +\infty]\), \(i=1,\dots,k\), be closed convex functions.  
	Suppose there exists \(\bar{x} \in \bigcap_{i=1}^k \mathrm{dom} \psi_i\) at which all but at most one of the functions are continuous.  
	Then for all \(x \in \bigcap_{i=1}^k \mathrm{dom} \psi_i\),
	\[
	\partial \left( \sum_{i=1}^k \psi_i \right)(x)
	= \sum_{i=1}^k \partial \psi_i(x)
	:= \left\{ \sum_{i=1}^k w_i \ \middle|\ w_i \in \partial \psi_i(x), \ i=1,\dots,k \right\}.
	\]
\end{theorem}

\section{Characterization of Solutions}
\label{characterization}

Building on the convex analysis framework established in Section~\ref{sec:preliminaries}, we now characterize the mathematical structure of solutions to the Generalized Heron--Waist Problem. This section establishes three fundamental results: existence of optimal solutions under mild boundedness conditions, uniqueness under additional geometric assumptions, and necessary and sufficient optimality conditions via subdifferential calculus. Together, these results provide a complete variational characterization of optimal configurations and form the theoretical foundation for algorithmic development.

\subsubsection*{Terminology and Connectivity Structure}

To formulate the existence theory for the Generalized Heron--Waist Problem, we introduce
some notions describing how the constraint sets are interconnected through the
weights in the objective function.

For the index set $J \subseteq \{ 1,2, \ldots, m\}$, let $\mathcal{C} := \{\, C_j \mid j \in J \,\} \cup \{\mathcal{S}\}$
denote a collection of constraint sets associated with the chain variables and,
optionally, the hub variable. 

\begin{enumerate}[label=(\Roman*)]
	\item \textbf{Polygonal chain of sets with hub.}  
	If $\rho_i > 0$ and $\omega_i > 0$ for all $i = 1,\ldots,m$, then the ordered collection
	$\{ C_1, C_2, \ldots, C_m, \mathcal{S} \}$
	is called a \emph{polygonal chain of sets with hub $\mathcal{S}$}.  
	In this case, each chain set is coupled both to its neighbors and to the hub through
	active interaction weights.
	
	\item \textbf{Polygonal chain of sets without hub.}  
	If $\omega_i = 0$ for all $i = 1,\ldots,m$ but $\rho_i > 0$, the ordered collection
	$\{ C_1, C_2, \ldots, C_m \}$
	is called a \emph{polygonal chain of sets}.  
	Here, the geometry is governed solely by interactions between neighboring chain sets,
	with no coupling to the hub.
	
	\item \textbf{Connected component with hub.}  Let $K \subseteq \{1,2,\ldots,m\}$. An ordered subcollection
	$\mathcal{C}_K := \{ C_k \mid k \in K \} \cup \{ \mathcal{S} \}$
	is called a \emph{connected component of a polygonal chain with hub $\mathcal{S}$}
	if the following conditions hold:
	
	\begin{enumerate}[label=(\alph*)]
		\item There exists at least one $k \in K$ such that $\omega_k > 0$, ensuring a nontrivial coupling with the hub.
		
		\item For every $k \in K$, either $C_k$ is directly connected to the hub, that is, $\omega_k > 0$, or, if $\omega_k = 0$, then $C_k$ is connected to a neighboring set through an active chain interaction that eventually links it to the hub. More precisely, if $\omega_k = 0$ and $\rho_{k-1} = 0$ but $\rho_k > 0$, then there exists an index $j$ with $k < j \le l$ such that
		$\rho_{k+1} > 0, \rho_{k+2} > 0, \ldots, \rho_l > 0$, $\rho_{l+1} = 0$, and $\omega_j > 0$.
	\end{enumerate}
	
This definition captures the notion that information propagates through the chain via positive weights and possibly through the hub.

\begin{figure}[H]
	\centering
	\includegraphics[width=0.75\textwidth]{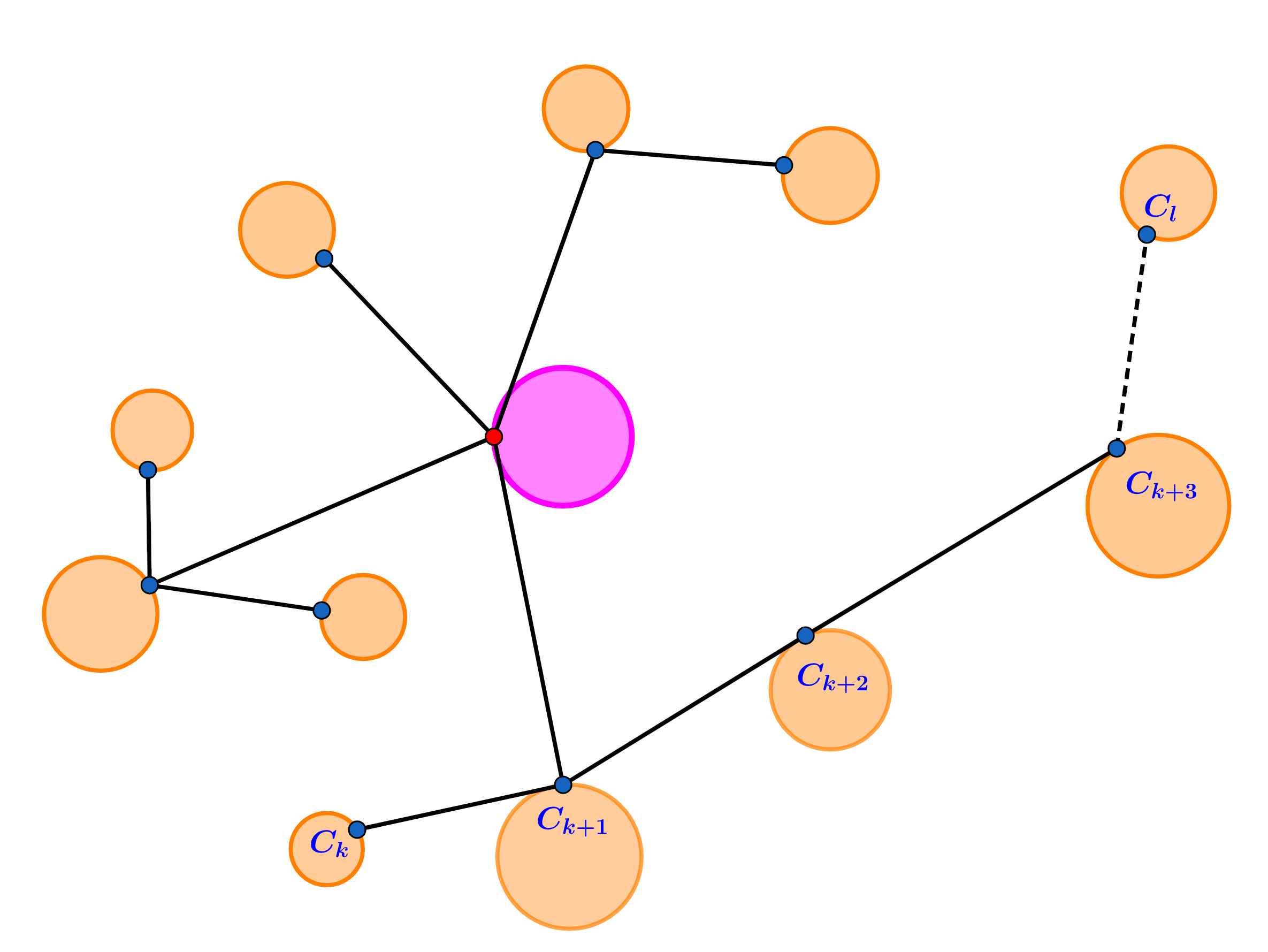}
	\caption{Illustration of a connected component of a polygonal chain with hub $\mathcal{S}$}
\end{figure}

	\item \textbf{Connected component without hub.}  
	An ordered subcollection $\{ C_k, C_{k+1}, \ldots, C_{k+t} \}$
	is called a \emph{connected component of a polygonal chain of sets} if
	$	\rho_{k+j} > 0 \quad \text{for all } j = 0,1,\ldots,t-1.$
	we assume that every component, in this case,
	contains at least two $C_j$'s, so that no decoupling occurs for any $C_j$. 	In particular, if $\omega_i = 0$ and $\rho_i > 0$ for all $i = 1,\ldots,m$, then the
	entire collection $\{ C_1, \ldots, C_m \}$ forms a single connected component. 
	
	\item \textbf{Edge length.}  
	For any adjacent chain variables $a_j$ and $a_{j+1}$, the quantity
	$	\|a_j - a_{j+1}\|$
	is referred to as the \emph{edge length} associated with the corresponding pair of
	sets $C_j$ and $C_{j+1}$.
\end{enumerate}

These notions allow us to precisely characterize how boundedness and interaction
structure propagate across the chain, which is essential for establishing existence
of optimal solutions in the presence of vanishing weights.

\begin{definition}\label{def:general_position}
	Let $C_1, C_2, \ldots, C_m, \mathcal{S} \subset \mathbb{R}^n$ be nonempty, closed, and convex sets.
	For each $i \in \{1, \ldots, m\}$, define the convex hull
	\[
	K_i := \operatorname{conv}\!\bigl(
	C_1 \cup \cdots \cup C_{i-1} \cup C_{i+1} \cup \cdots \cup C_m \cup \mathcal{S}
	\bigr).
	\]
	We say that the collection $\{C_i\}_{i=1}^m$ is in \emph{general position} (with respect to $\mathcal{S}$)
	if $C_i \cap K_i = \varnothing$ for all  $i = 1, \ldots, m.$
\end{definition}

\subsection{Existence of Optimal Solutions}
\label{subsec:existence}

We establish that the Generalized Heron--Waist Problem admits at least one optimal solution under natural boundedness assumptions on the constraint sets. The key insight is that the vertices connected to the hub with positive weight induce boundedness through radial coupling, while purely perimeter-connected components require at least one bounded constraint set in each component to prevent unbounded configurations.

\begin{theorem}\label{existencetheorem}
	Let $C_1, \ldots, C_m, \mathcal{S} \subseteq \mathbb{R}^n$ be nonempty, closed, and
	convex sets. Assume that the collection $\{C_i\}_{i=1}^m$ is in general position in the sense of
	Definition~\ref{def:general_position}.Let $\rho_i \ge 0$ and $\omega_i \ge 0$ for $i = 1,\ldots,m$, and
	assume the following conditions hold:
	
	\begin{enumerate}[label=(\alph*)]
		\item For every connected component of the polygonal chain of sets with hub $\mathcal{S}$, at least one of the sets $\mathcal{S}$ or $C_j$ is bounded in that
		component (here we assume that $\omega_j > 0$ for at least one $j$, i.e., no
		decoupling of $\mathcal{S}$ with $C_j$ occurs).
		
		\item For every connected component of the polygonal chain without hub $\mathcal{S}$ i.e. with $C_j$'s only,
		at least one $C_j$ is bounded.
	\end{enumerate}
	
	Then the generalized Heron--waist problem (5) admits an optimal solution.
\end{theorem}

\begin{proof}
	The feasible region $\Omega = C_1 \times C_2 \times \cdots \times C_m \times \mathcal{S}$
	is nonempty by assumption and closed, being the Cartesian product of closed sets.
	The objective function $J : \Omega \to \mathbb{R}$ is continuous, being a finite
	sum of Euclidean norms, and satisfies $J(u) \ge 0$ for all $u \in \Omega$.
	Define
	\[
	\nu := \inf_{u \in \Omega} J(u),
	\]
	since $J(u)$ is not identically $+\infty$, $0 \le \nu < \infty$.
	
	By definition of infimum, there exists a minimizing sequence
	\[
	\{u^{(\ell)}\}_{\ell=1}^{\infty} \subset \Omega,
	\qquad
	u^{(\ell)} = \bigl(a_1^{(\ell)}, a_2^{(\ell)}, \ldots, a_m^{(\ell)}, x^{(\ell)}\bigr),
	\]
	satisfying
	\[
	\lim_{\ell \to \infty} J\bigl(u^{(\ell)}\bigr) = \nu.
	\]

	Since a convergent sequence is bounded, there exists $M > 0$ such that
	\[
	J\bigl(u^{(\ell)}\bigr) \le M
	\quad \text{for all } \ell.
	\]
	
	We now demonstrate that $\{u^{(\ell)}\}$ is bounded in $\mathbb{R}^{n \times (m+1)}$.

	\medskip
	\noindent
	\textbf{Case (I):} For the connected component with the hub set $\mathcal{S}$.
	
	If the set $\mathcal{S}$ is bounded, then $\{x^{(\ell)}\}$ is bounded.
	Otherwise, by assumption (a), some $C_j$ is bounded, which implies the boundedness
	of $\{a_j^{(\ell)}\}$. Together with bounded edge length\\
	
	\[
	\left\| a_j^{(\ell)} - x^{(\ell)} \right\|
	\le \frac{J\bigl(u^{(\ell)}\bigr)}{\omega_j}
	\le \frac{M}{\omega_j}, 
	\quad \text{if } \omega_j > 0.
	\]	
and/or 
	\[
	\|a_j^{(\ell)} - a_{j+1}^{(\ell)}\|
	\le \frac{J\bigl(u^{(\ell)}\bigr)}{\rho_j}
	\le \frac{M}{\rho_j},
	\quad \text{if } \rho_j > 0.
	\]
	and/or
	\[
	\|a_{j-1}^{(\ell)} - a_j^{(\ell)}\|
	\le \frac{M}{\rho_{j-1}}, 
	\quad \text{if } \rho_{j-1} > 0.
	\]
	Since at least one of $\omega_j, \rho_j$ and $\rho_{j-1}$ is positive, and $\omega_k>0$ for at least one $k$, we can use 
	
	\[
	\left\| a_k^{(\ell)} - x^{(\ell)} \right\|
	\le \frac{J\bigl(u^{(\ell)}\bigr)}{\omega_k}
	\le \frac{M}{\omega_k}, 
	\]	
	
	The repeated application of triangle inequality implies boundedness of all
	sequences $a_j^{(\ell)}$ and $x^\ell$ of chain sets in that component.

	Thus for a connected component having the set $\mathcal{S}$, by assumption (a),
	all the corresponding sequences in chain sets and hub set are bounded.

	\medskip
	\noindent
	\textbf{Case (II):} If the component has only $C_j$'s, then by assumption (b)
	at least one $C_j$ is bounded.
	
	As discussed in (a), we see that all the sequences
	$\{a_j^{(\ell)}\}$ are bounded for all sets $C_j$ in the component.
	Thus, combining Case (I) and Case (II), we conclude that all sequences
	$\{a_i^{(\ell)}\}$, $i = 1,2,\ldots,m$, and $\{x^{(\ell)}\}$ are bounded.

	By the Bolzano--Weierstrass theorem, there exists a convergent subsequence $\{\mathbf{u}^{(\ell_k)}\}$ with limit
	\begin{equation}
		\label{eq:limit_point}
		\mathbf{u}^* = \lim_{k \to \infty} \mathbf{u}^{(\ell_k)} \in \mathbb{R}^{n(m+1)}.
	\end{equation}
	
	Since each $\mathbf{u}^{(\ell_k)} \in \Omega$ and $\Omega$ is closed, we have $\mathbf{u}^* \in \Omega$. By continuity of $J$,
	\begin{equation}
		\label{eq:optimal_value}
		J(\mathbf{u}^*) = \lim_{k \to \infty} J\bigl(\mathbf{u}^{(\ell_k)}\bigr) = \nu = \inf_{\mathbf{u} \in \Omega} J(\mathbf{u}),
	\end{equation}
	proving that $\mathbf{u}^*$ is an optimal solution.
	
\end{proof}

\begin{corollary}[Uniform Positive Radial Weights]
	\label{cor:uniform_positive}
	If $\omega_i > 0$ for all $i \in \{1, \ldots, m\}$ and either $\mathcal{S}$ or at least one $C_i$ is bounded, then the GHWP admits an optimal solution.
\end{corollary}

\begin{proof}
	In this case, condition (b) is vacuous and (a) is assumed by hypothesis.
\end{proof}

\begin{corollary}[Bounded Feasible Region]
	\label{cor:bounded_feasible}
If $\Omega$ is bounded, i.e., all the sets are bounded, then the GHWP admits an optimal solution regardless of the weight configuration.
\end{corollary}

\begin{proof}
	Boundedness of $\Omega$ immediately implies boundedness of every minimizing sequence. The Bolzano--Weierstrass theorem then guarantees existence of a convergent subsequence whose limit, by continuity and closedness, is optimal.
\end{proof}

\subsection{Uniqueness of Optimal Solutions}
\label{subsec:uniqueness}

Having established the existence of optimal solutions, we now turn to the question of uniqueness. Although convexity ensures that every local minimizer is a global one, uniqueness generally requires additional geometric structures. In the generalized Heron–Waist problem, the combination of strict convexity of the constraint sets, strictly positive interaction weights, and suitable geometric separation plays a crucial role in obtaining a unique solution. In this subsection, we develop supporting remarks and lemmas that lead to a sufficient criterion guaranteeing uniqueness of the optimal configuration.

\medskip
Consider the generalized Heron problem  (as discussed in ~\cite{Mordukhovich2012}):  
\[
\min_{x \in S} D(x) := \sum_{i=1}^{n} d(x, C_i),
\]
where $S$ and $C_i$ are nonempty, closed, and convex subsets of $\mathbb{R}^n$ satisfying $S \cap C_i = \varnothing$ for all $i=1,\ldots,n$.

For any $\bar{x} \in S$, define the normalized direction vectors
\[
a_i(\bar{x}) := \frac{\bar{x} - \operatorname{proj}_{C_i}(\bar{x})}{d(\bar{x}, C_i)},
\qquad i = 1,\ldots,n,
\]
which are well-defined and nonzero due to the strict separation of $S$ and $C_i$. According to \cite[Theorem~3.2]{Mordukhovich2012}, the point $\bar{x}$ is optimal solution if and only if  
\begin{equation}
	\label{eq:borwein_condition}
	-\sum_{i=1}^{n} a_i(\bar{x}) \in N_S(\bar{x})
\end{equation}
holds, where $N_S(\bar{x})$ denotes the normal cone to $S$ at $\bar{x}$.

\begin{remark}\label{rem:boundary_characterization}
	Consider the GHP, assume that each set $C_i$ is a singleton $\{x_i\}$ and that the collection $\{x_i\}_{i=1}^{n}$ is (strictly) separated from $S$ by a hyperplane. In this situation, all points $x_i$ lie on one side of the hyperplane, while the set $S$ lies entirely on the opposite side. In this case,
	\[
	\sum_{i=1}^{n} a_i(\bar{x}) = \sum_{i=1}^{n} \frac{\bar{x} - x_i}{\|\bar{x} - x_i\|} \neq 0.
	\]
	If $\bar{x} \in \operatorname{int}(S)$, then the normal cone satisfies $N_S(\bar{x}) = \{0\}$, which contradicts the optimality condition \eqref{eq:borwein_condition}. Therefore, any optimal solution must necessarily lie on the boundary of the feasible region, that is,
	\[
	-\sum_{i=1}^{n} a_i(\bar{x}) \in N_S(\bar{x})
	\quad \Longleftrightarrow \quad
	\bar{x} \in \operatorname{bd}(S).
	\]
	In particular, if $\bar{x} \in \operatorname{int}(S)$, then
	\[
	D(\bar{x}) > \min_{x \in S} D(x).
	\]
\end{remark}

\begin{theorem}\label{thm:uniqueness}
	Let $C_1, C_2, \ldots, C_m, \mathcal{S} \subseteq \mathbb{R}^n$ be nonempty, closed, and
	strictly convex sets. Assume that the collection $\{C_i\}_{i=1}^m$ is in general
	position with respect to the hub set $\mathcal{S}$ in the sense of
	Definition~\ref{def:general_position}, and at least one of the sets
	$C_1, \ldots, C_m, \mathcal{S}$ is bounded.	Furthermore, suppose that all interaction weights are strictly positive, that is, $	\rho_i > 0$ and $ \omega_i > 0$ for all $i = 1, \ldots, m.$ Then the Generalized Heron--Waist Problem~\eqref{Heron_Waist_Problem} admits a
	\emph{unique} optimal solution.
\end{theorem}

\begin{proof}
	
		The objective function $J : \Omega \to \mathbb{R}$ is convex and continuous,
	and the feasible region $\Omega = C_1 \times C_2 \times \cdots \times C_m \times \mathcal{S}$ is nonempty, closed, and convex. Moreover, by assumption, at least one of the sets $C_i$ or $\mathcal{S}$ is bounded.
	Hence, by Theorem~\ref{existencetheorem}, an optimal solution to the problem exists.
	
	We now prove uniqueness by contradiction. Assume that there exist two distinct optimal solutions
	$	\mathbf{U} = (a_1, a_2, \ldots, a_m, x)$ and $	\mathbf{W} = (b_1, b_2, \ldots, b_m, y),$
 $\mathbf{U} \neq \mathbf{W},$	such that $J(\mathbf{U}) = J(\mathbf{W}) =\nu.$ Since the feasible region $\Omega$ is convex, their midpoint 
	\[
	\mathbf{M} := \frac{1}{2}(\mathbf{U} + \mathbf{W}) =(t_1,t_2, \ldots, t_m, z)
	\]
	also belongs to $\Omega$, where $\displaystyle t_i= \frac{(a_i+b_i)}{2}, z=\frac{x+y}{2}$. By convexity of the objective function $J$, Jensen’s
	inequality yields
	\[
	J(\mathbf{M})
	\le
	\frac{1}{2} J(\mathbf{U}) + \frac{1}{2} J(\mathbf{W})
	= \nu.
	\]
	Since $\nu$ is the minimal value, we must have $J(\mathbf{M}) = \nu$,
	and hence $\mathbf{M}$ is also an optimal solution.
	
	Because $\mathbf{U} \neq \mathbf{W}$, there exists at least one index $k$ such that
	$a_k \neq b_k$ or $x \neq y$. By strict convexity of $C_k$ or $\mathcal{S}$, the
	corresponding midpoint component satisfies
	\[
	t_k=\frac{a_k + b_k}{2} \in \operatorname{int}(C_k)
	\quad \text{or} \quad
	z=\frac{x + y}{2} \in \operatorname{int}(\mathcal{S}).
	\]
	For any point in the interior of a convex set, the normal cone reduces to the zero
	vector. Consequently,
	\[
	N_{C_k}\!\left(t_k\right) = \{\mathbf{0}\}
	\quad \text{or} \quad
	N_{\mathcal{S}}\!\left(z\right) = \{\mathbf{0}\}.
	\]
	
	On the other hand, since $\mathbf{M}$ is optimal, it must satisfy the first-order
	optimality condition
	\[
	\mathbf{0} \in \partial J(\mathbf{M}) + N_{\Omega}(\mathbf{M}),
	\]
	which holds componentwise.Without loss of generality assume that for the index $k$, $\displaystyle t_k =\frac{a_k+b_k}{2}$ $\in \operatorname{int}(C_k).$ In this case we obtain


	\begin{equation}\label{GHWPzerosum}
			\mathbf{0} = g_k + n_{k},
	\end{equation}
	where $g_k \in \partial_{t_k} J(\mathbf{M})$ and
	$n_{k} \in N_{C_k}(t_k)$.
	
	Since $t_k$ lies in the interior of $C_k$, we have $n_{k} = \mathbf{0}$,
	and hence $g_k = \mathbf{0}$. However, because all weights are strictly positive,
	the $k$-th subgradient of $J$ at $\mathbf{M}$ is given by
	\[
	g_k
	=
	\rho_{k-1}
	\frac{t_k - t_{k-1}}{\|t_k - t_{k-1}\|}
	+
	\rho_k
	\frac{t_k - t_{k+1}}{\|t_k - t_{k+1}\|}
	+
	\omega_k
	\frac{t_k - z}{\|t_k - z\|}.
	\]
	This expression is a weighted sum of unit vectors with strictly positive coefficients. Since sets are in general position, therefore by Theorem \ref{seperationtheorem}, $C_k$ and $\operatorname{conv}(C_{k-1},C_{k+1},S)$ can be separated by a hyperplane. In fact, the convex hull of $\{t_{k-1}, t_{k+1}, z\}$ can be strictly separated from the set $C_k$. Invoking Remarks \ref{rem:boundary_characterization} and \eqref{GHWPzerosum} such a sum cannot vanish, i.e., $g_k \neq 0$.
	This contradiction shows that two distinct optimal solutions cannot exist.
	Therefore, the optimal solution is unique.
\end{proof}

\subsection{Optimality Conditions}
\label{subsec:optimality_conditions}

We established existence and uniqueness under reasonable conditions, we now derive a complete characterization of optimal solutions via first-order optimality conditions. These conditions, expressed through subdifferential inclusions and normal cone equilibrium, provide both necessary and sufficient criteria for global optimality and admit elegant geometric interpretations as force-balance equations at each variable.

\begin{theorem}\label{thm:optimality_conditions}
Let $C_1, C_2, \ldots, C_m, \mathcal{S} \subset \mathbb{R}^n$ be nonempty, closed, and convex sets, which are pairwise disjoint that is $C_i \cap C_j = \emptyset~ \forall i \neq j$ and $C_i \cap \mathcal{S} = \emptyset~ \forall i$. Let $\boldsymbol{\rho} = (\rho_1, \ldots, \rho_m)$ and $\boldsymbol{\omega} = (\omega_1, \ldots, \omega_m)$ be nonnegative weight vectors. Then $\mathbf{u}^* = (a_1^*, a_2^*, \ldots, a_m^*, x^*) \in \Omega$ 	is a global optimal solution to the Generalized Heron--Waist Problem \eqref{Heron_Waist_Problem} if and only if there exist normal vectors
	\[
	n_{C_i} \in N_{C_i}(a_i^*) \quad \text{for each } i = 1, \ldots, m,
	\quad \text{and} \quad
	n_{\mathcal{S}} \in N_{\mathcal{S}}(x^*)
	\]
	satisfying the following equilibrium conditions:
	
	\medskip
	
	\noindent\textbf{(O1) Chain Variable Equilibrium:} For each $i = 1, 2, \ldots, m$,
	\begin{equation}
		\label{eq:chain_equilibrium}
		\rho_{i-1} \frac{a_i^* - a_{i-1}^*}{\|a_i^* - a_{i-1}^*\|}
		+ \rho_i \frac{a_i^* - a_{i+1}^*}{\|a_i^* - a_{i+1}^*\|}
		+ \omega_i \frac{a_i^* - x^*}{\|a_i^* - x^*\|}
		+ n_{C_i}
		= \mathbf{0},
	\end{equation}
	where indices are cyclic: $a_0^* := a_m^*$ and $a_{m+1}^* := a_1^*$.
	
	\medskip
	
	\noindent\textbf{(O2) Hub Variable Equilibrium:}
	\begin{equation}
		\label{eq:hub_equilibrium}
		\sum_{i=1}^{m} \omega_i \frac{x^* - a_i^*}{\|x^* - a_i^*\|} + n_{\mathcal{S}} = \mathbf{0}.
	\end{equation}
	
	\medskip
	
	\noindent\textbf{(O3) Global Normal Balance:}
	\begin{equation}
		\label{eq:global_balance}
		\sum_{i=1}^{m} n_{C_i} + n_{\mathcal{S}} = \mathbf{0}.
	\end{equation}
\end{theorem}

\begin{proof}

	We reformulate this generalized Heron-Waist  constrained problem as an unconstrained problem by using the indicator function. The feasible region
$\Omega := C_1 \times C_2 \times \cdots \times C_m \times \mathcal{S}$
is nonempty, closed, and convex. Introducing the indicator function $\delta_{\Omega}$, the problem becomes
\begin{equation}
	\label{eq:unconstrained_form}
	\min_{\mathbf{u} \in \mathbb{R}^{n(m+1)}} \Phi(\mathbf{u}),
	\quad \text{where} \quad
	\Phi(\mathbf{u}) := J(\mathbf{u}) + \delta_{\Omega}(\mathbf{u}).
\end{equation}

Since both $J$ and $\delta_{\Omega}$ are convex, $\Phi$ is convex. By Fermat's rule in convex analysis, $\mathbf{u}^*$ is a global minimizer if and only if
\begin{equation}
	\label{eq:fermat_rule}
	\mathbf{0} \in \partial \Phi(\mathbf{u}^*).
\end{equation}

By the subdifferential sum rule (Theorem~\ref{sumofsubgradient}),
\begin{equation}
	\label{eq:subdiff_sum}
	\partial \Phi(\mathbf{u}^*) = \partial J(\mathbf{u}^*) + \partial \delta_{\Omega}(\mathbf{u}^*).
\end{equation}

Recall that, the subdifferential of the indicator function equals the normal cone:
\begin{equation}
	\label{eq:indicator_normal}
	\partial \delta_{\Omega}(\mathbf{u}^*) = N_{\Omega}(\mathbf{u}^*).
\end{equation}

By the product structure of $\Omega$ and Proposition~\ref{Normalcone_Product},
\begin{equation}
	\label{eq:normal_product}
	N_{\Omega}(\mathbf{u}^*) = N_{C_1}(a_1^*) \times N_{C_2}(a_2^*) \times \cdots \times N_{C_m}(a_m^*) \times N_{\mathcal{S}}(x^*).
\end{equation}

Therefore, any element $\mathbf{n} \in N_{\Omega}(\mathbf{u}^*)$ has the block structure
\begin{equation}
	\label{eq:normal_structure}
	\mathbf{n} = (n_{C_1}, n_{C_2}, \ldots, n_{C_m}, n_{\mathcal{S}}),
\end{equation}
where $n_{C_i} \in N_{C_i}(a_i^*)$ for $i = 1, \ldots, m$ and $n_{\mathcal{S}} \in N_{\mathcal{S}}(x^*)$.

We now compute $\partial J(\mathbf{u}^*)$. The objective function is
\[
J(\mathbf{u}) = \sum_{i=1}^{m} \rho_i \|a_i - a_{i+1}\| + \sum_{i=1}^{m} \omega_i \|a_i - x\|.
\]

For each term in the objective, we must compute its subdifferential with respect to the composite variable $\mathbf{u} = (a_1, \ldots, a_m, x)$. Consider first a perimeter term $\rho_i \|a_i - a_{i+1}\|$ for a fixed index $i$. By Proposition~\ref{subgradientformula}, Remark~\ref{remark:subgradient_equality}  and Example~\ref{ex:subdiff_distance}. Since by assumption $a_i^* \neq a_{i+1}^*$, the subdifferential of the Euclidean norm $\|a_i - a_{i+1}\|$ with respect to the composite variable $\mathbf{u}$ has the form of a vector in $\mathbb{R}^{n(m+1)}$ with nonzero components only in the positions corresponding to $a_i$ and $a_{i+1}$:
\begin{equation}
	\label{eq:perimeter_subdiff}
	\partial_{a_i} [\|a_i - a_{i+1}\|] = \left( \mathbf{0}, \ldots, \mathbf{0}, \underbrace{\frac{a_i - a_{i+1}}{\|a_i - a_{i+1}\|}}_{i\text{-th position}}, \underbrace{\frac{a_{i+1} - a_i}{\|a_{i+1} - a_i\|}}_{(i+1)\text{-th position}}, \mathbf{0}, \ldots, \mathbf{0} \right).
\end{equation}

Multiplying by the weight $\rho_i \geq 0$, the subdifferential of the $i$-th perimeter term becomes
\begin{equation}
	\label{eq:weighted_perimeter}
	\partial_{a_i} [\rho_i \|a_i - a_{i+1}\|] = \left( \mathbf{0}, \ldots, \mathbf{0}, \rho_i \frac{a_i - a_{i+1}}{\|a_i - a_{i+1}\|}, \rho_i \frac{a_{i+1} - a_i}{\|a_{i+1} - a_i\|}, \mathbf{0}, \ldots, \mathbf{0} \right).
\end{equation}

Similarly, for a radial term $\omega_i \|a_i - x\|$, assuming $a_i^* \neq x^*$, the subdifferential has nonzero components in the position corresponding to $a_i$ and in the final position corresponding to the hub variable $x$:
\begin{equation}
	\label{eq:radial_subdiff}
	\partial_x [\omega_i \|a_i - x\|] = \left( \mathbf{0}, \ldots, \mathbf{0}, \underbrace{\omega_i \frac{a_i - x}{\|a_i - x\|}}_{i\text{-th position}}, \mathbf{0}, \ldots, \mathbf{0}, \underbrace{\omega_i \frac{x - a_i}{\|x - a_i\|}}_{\text{hub position}} \right).
\end{equation}

Consequently,
\begin{align}
	\label{eq:full_subdiff_expansion}
	\partial J(\mathbf{u}^*) 
	&= \partial \left[ \sum_{i=1}^{m} \rho_i \|a_i - a_{i+1}\| + \sum_{i=1}^{m} \omega_i \|a_i - x\| \right] \notag \\[6pt]
	&= \rho_1 \left( \frac{a_1^* - a_2^*}{\|a_1^* - a_2^*\|}, \frac{a_2^* - a_1^*}{\|a_2^* - a_1^*\|}, \mathbf{0}, \ldots, \mathbf{0} \right) \notag \\[6pt]
	&\quad + \rho_2 \left( \mathbf{0}, \frac{a_2^* - a_3^*}{\|a_2^* - a_3^*\|}, \frac{a_3^* - a_2^*}{\|a_3^* - a_2^*\|}, \mathbf{0}, \ldots, \mathbf{0} \right) \notag \\[6pt]
	&\quad + \rho_3 \left( \mathbf{0}, \mathbf{0}, \frac{a_3^* - a_4^*}{\|a_3^* - a_4^*\|}, \frac{a_4^* - a_3^*}{\|a_4^* - a_3^*\|}, \mathbf{0}, \ldots, \mathbf{0} \right) \notag \\[6pt]
	&\quad + \cdots \notag \\[6pt]
	&\quad + \rho_m \left( \frac{a_m^* - a_1^*}{\|a_m^* - a_1^*\|}, \mathbf{0}, \ldots, \mathbf{0}, \frac{a_1^* - a_m^*}{\|a_1^* - a_m^*\|} \right) \notag \\[6pt]
	&\quad + \omega_1 \left( \frac{a_1^* - x^*}{\|a_1^* - x^*\|}, \mathbf{0}, \ldots, \mathbf{0}, \frac{x^* - a_1^*}{\|x^* - a_1^*\|} \right) \notag \\[6pt]
	&\quad + \omega_2 \left( \mathbf{0}, \frac{a_2^* - x^*}{\|a_2^* - x^*\|}, \mathbf{0}, \ldots, \mathbf{0}, \frac{x^* - a_2^*}{\|x^* - a_2^*\|} \right) \notag \\[6pt]
	&\quad + \cdots \notag \\[6pt]
	&\quad + \omega_m \left( \mathbf{0}, \ldots, \mathbf{0}, \frac{a_m^* - x^*}{\|a_m^* - x^*\|}, \frac{x^* - a_m^*}{\|x^* - a_m^*\|} \right).
\end{align}

Let $\mathbf{0}$ be represented as a $(m+1)$ tuple of zero vectors of dimension $n$:
\[
\mathbf{0} = (\mathbf{0}_n, \mathbf{0}_n, \ldots, \mathbf{0}_n) \in \mathbb{R}^{n \times (m+1)}.
\]

The optimality condition \eqref{eq:fermat_rule} now requires that for some $(n_{C_1}, n_{C_2}, \ldots, n_{C_m}, n_{\mathcal{S}})$ $\in N_\Omega(\mathbf{u}^*)$
\begin{equation}
	\label{eq:optimality_with_normal}
\mathbf{0} \in	\partial J(\mathbf{u}^*) + (n_{C_1}, n_{C_2}, \ldots, n_{C_m}, n_{\mathcal{S}}),
\end{equation}
or equivalently,
\begin{equation}
	\label{eq:optimality_equation}
\mathbf{0} =	\partial J(\mathbf{u}^*) + (n_{C_1}, n_{C_2}, \ldots, n_{C_m}, n_{\mathcal{S}}).
\end{equation}

Substituting \eqref{eq:full_subdiff_expansion} into the optimality condition \eqref{eq:optimality_equation}, we obtain the explicit tuple equality:
\begin{equation}
	\label{eq:tuple_equality_full}
	\begin{aligned}
		(\mathbf{0}_n, \mathbf{0}_n, \ldots, \mathbf{0}_n, \mathbf{0}_n) = \bigg(
		&\rho_m \frac{a_1^* - a_m^*}{\|a_1^* - a_m^*\|} + \rho_1 \frac{a_1^* - a_2^*}{\|a_1^* - a_2^*\|} + \omega_1 \frac{a_1^* - x^*}{\|a_1^* - x^*\|}, \\[4pt]
		&\rho_1 \frac{a_2^* - a_1^*}{\|a_2^* - a_1^*\|} + \rho_2 \frac{a_2^* - a_3^*}{\|a_2^* - a_3^*\|} + \omega_2 \frac{a_2^* - x^*}{\|a_2^* - x^*\|}, \\[4pt]
		&\qquad \vdots \\[4pt]
		&\rho_{m-1} \frac{a_m^* - a_{m-1}^*}{\|a_m^* - a_{m-1}^*\|} + \rho_m \frac{a_m^* - a_1^*}{\|a_m^* - a_1^*\|} + \omega_m \frac{a_m^* - x^*}{\|a_m^* - x^*\|}, \\[4pt]
		&\sum_{i=1}^{m} \omega_i \frac{x^* - a_i^*}{\|x^* - a_i^*\|}
		\bigg) \\[6pt]
		&\qquad\qquad + (n_{C_1}, n_{C_2}, \ldots, n_{C_m}, n_{\mathcal{S}}).
	\end{aligned}
\end{equation}

For this tuple equality to hold, each component must individually equal the zero vector $\mathbf{0}_n \in \mathbb{R}^n$. This componentwise requirement yields the following system of equilibrium equations.

\medskip
\noindent\textbf{Component 1 ($a_1^*$ equilibrium):}
\[
\mathbf{0}_n = \rho_m \frac{a_1^* - a_m^*}{\|a_1^* - a_m^*\|} + \rho_1 \frac{a_1^* - a_2^*}{\|a_1^* - a_2^*\|} + \omega_1 \frac{a_1^* - x^*}{\|a_1^* - x^*\|} + n_{C_1}.
\]

\noindent\textbf{Component 2 ($a_2^*$ equilibrium):}
\[
\mathbf{0}_n = \rho_1 \frac{a_2^* - a_1^*}{\|a_2^* - a_1^*\|} + \rho_2 \frac{a_2^* - a_3^*}{\|a_2^* - a_3^*\|} + \omega_2 \frac{a_2^* - x^*}{\|a_2^* - x^*\|} + n_{C_2}.
\]

\noindent\textbf{In general, component $i$ ($a_i^*$ equilibrium):}
For each $i = 1, 2, \ldots, m$, we have
\begin{equation}
	\label{eq:chain_equilibrium_derived}
	\mathbf{0}_n = \rho_{i-1} \frac{a_i^* - a_{i-1}^*}{\|a_i^* - a_{i-1}^*\|} + \rho_i \frac{a_i^* - a_{i+1}^*}{\|a_i^* - a_{i+1}^*\|} + \omega_i \frac{a_i^* - x^*}{\|a_i^* - x^*\|} + n_{C_i},
\end{equation}
where indices are cyclic: $a_0^* := a_m^*$ and $a_{m+1}^* := a_1^*$. This is precisely condition \textbf{(O1)}.

\medskip
\noindent\textbf{Component $m+1$ (hub $x^*$ equilibrium):}
\begin{equation}
	\label{eq:hub_equilibrium_derived}
	\mathbf{0}_n = \sum_{i=1}^{m} \omega_i \frac{x^* - a_i^*}{\|x^* - a_i^*\|} + n_{\mathcal{S}}.
\end{equation}
This is precisely condition \textbf{(O2)}.

\medskip
\noindent\textbf{Deriving the global balance condition (O3):}
To establish the global equilibrium, we sum equations \eqref{eq:chain_equilibrium_derived} over all chain vertices $i = 1, \ldots, m$ and add equation \eqref{eq:hub_equilibrium_derived}. This yields:
\begin{equation}
	\label{eq:sum_all_equations}
	\sum_{i=1}^{m} \left[ \rho_{i-1} \frac{a_i^* - a_{i-1}^*}{\|a_i^* - a_{i-1}^*\|} + \rho_i \frac{a_i^* - a_{i+1}^*}{\|a_i^* - a_{i+1}^*\|} + \omega_i \frac{a_i^* - x^*}{\|a_i^* - x^*\|} + n_{C_i} \right] + \sum_{i=1}^{m} \omega_i \frac{x^* - a_i^*}{\|x^* - a_i^*\|} + n_{\mathcal{S}} = \mathbf{0}_n.
\end{equation}

We now observe crucial cancellations due to the geometric structure:

\begin{equation}
	\label{eq:perimeter_cancel}
	\sum_{i=1}^{m} \left[ \rho_i \frac{a_i^* - a_{i+1}^*}{\|a_i^* - a_{i+1}^*\|} + \rho_i \frac{a_{i+1}^* - a_i^*}{\|a_{i+1}^* - a_i^*\|} \right] = \mathbf{0}_n.
\end{equation}

\begin{equation}
	\label{eq:radial_cancel}
	\sum_{i=1}^{m} \omega_i \frac{a_i^* - x^*}{\|a_i^* - x^*\|} + \sum_{i=1}^{m} \omega_i \frac{x^* - a_i^*}{\|x^* - a_i^*\|} = \mathbf{0}_n.
\end{equation}

After these cancellations in \eqref{eq:sum_all_equations}, only the normal cone contributions remain:
\begin{equation}
	\label{eq:global_balance_derived}
	\sum_{i=1}^{m} n_{C_i} + n_{\mathcal{S}} = \mathbf{0}_n.
\end{equation}
This is precisely condition \textbf{(O3)}, expressing the global balance of constraint forces.

By Fermat's rule, in convex setting, this establishes that $\mathbf{u}^*$ is a global minimizer of $J$ over $\Omega$. Therefore, conditions (O1), (O2), and (O3) are both necessary and sufficient for global optimality.

\end{proof}

\section{Computational Algorithms}
\label{algorithms}

The optimality conditions derived in Section~\ref{subsec:optimality_conditions} provide a theoretical characterization of the solution structure for the Generalized Heron--Waist Problem. To approximate such optimal points, we now design an iterative procedure based on the subgradient information and projection operators. In the following, we present the \emph{Projected Subgradient Algorithm (PSA)} for our proposed problem \eqref{Heron_Waist_Problem}, and establish its convergence properties, and subsequently demonstrate its performance through various numerical illustrations.

\subsection{Convergence Analysis}
\label{subsec:convergence_analysis}

Before presenting the convergence theorem, we first establish that the subgradients of our objective function are uniformly bounded. This is a standard assumption in the analysis of subgradient-based methods (see, for example,~\cite[Assumption 3.2.1, p.~153]{bertsekas2015convex}). Formally, there exists a constant $G > 0$ such that $\|g^{(k)}\| \le G$ for all iterations $k$,  where $g^{(k)} \in \partial J(u^{(k)})$. This assumption guarantees the stability of the projected subgradient updates and is automatically satisfied when the objective function is Lipschitz continuous with Lipschitz constant~$G$.

\begin{remark}[Boundedness of Subgradients]
	\label{rem:subgrad_bound}
	For our Generalized Heron--Waist Problem \eqref{eq:objective_ghwp}, this boundedness can be verified explicitly. Consider the objective function
	\[
	J(\mathbf{u}) = \sum_{i=1}^{m} \rho_i \|a_i - a_{i+1}\| + \sum_{i=1}^{m} \omega_i \|a_i - x\|, \quad \mathbf{u} = (a_1, \ldots, a_m, x).
	\]
	
	At each iteration $k$, the vector $g^{(k)}$ is a subgradient of the convex function $J$ at $\mathbf{u}^{(k)}$, given by
	\[
	g^{(k)} = \big(g_{a_1}^{(k)}, g_{a_2}^{(k)}, \ldots, g_{a_m}^{(k)}, g_x^{(k)}\big),
	\]
	where each component is computed according to the subdifferential calculus established in Theorem~\ref{thm:optimality_conditions}. Specifically, for each chain variable $a_i$ at iteration $k$:
	\begin{equation}
		\label{eq:ga}
		g_{a_i}^{(k)} =
		\begin{cases}
			\rho_{i-1}\dfrac{a_i^{(k)} - a_{i-1}^{(k)}}{\|a_i^{(k)} - a_{i-1}^{(k)}\|}
			+ \rho_i\dfrac{a_i^{(k)} - a_{i+1}^{(k)}}{\|a_i^{(k)} - a_{i+1}^{(k)}\|} \\[0.8em]
			\qquad + \omega_i\dfrac{a_i^{(k)} - x^{(k)}}{\|a_i^{(k)} - x^{(k)}\|},
			& \text{if all denominators} \neq 0, \\[0.8em]
			\mathbf{0}_n, & \text{otherwise},
		\end{cases}
	\end{equation}
	for $i = 1, \ldots, m$, with cyclic indexing $a_0^{(k)} := a_m^{(k)}$ and $a_{m+1}^{(k)} := a_1^{(k)}$. For the hub variable $x$:
	\begin{equation}
		\label{eq:gx}
		g_x^{(k)} = \sum_{i=1}^{m} \omega_i 
		\begin{cases}
			\dfrac{x^{(k)} - a_i^{(k)}}{\|x^{(k)} - a_i^{(k)}\|}, & \text{if } x^{(k)} \neq a_i^{(k)}, \\[0.6em]
			\mathbf{0}_n, & \text{otherwise}.
		\end{cases}
	\end{equation}
	
	Since each term \eqref{eq:ga} and \eqref{eq:gx} $\frac{a_i - a_{i-1}}{\|a_i - a_{i-1}\|}$ is a unit vector or zero vector, and similarly for all other terms, we can bound each component. For the chain variables, each $g_{a_i}^{(k)}$ is a weighted sum of at most three unit vectors:
	\[
	\|g_{a_i}^{(k)}\| \le \rho_{i-1} + \rho_i + \omega_i \le \max_{1 \le j \le m}(\rho_j + \rho_{j+1}) + \max_{1 \le j \le m}\omega_j.
	\]
	For the hub variable:
	\[
	\|g_x^{(k)}\| \le \sum_{i=1}^{m} \omega_i = \|\boldsymbol{\omega}\|_1,
	\]
	where $\|\boldsymbol{\omega}\|_1$ denotes the $\ell^1$-norm of the weight vector $\boldsymbol{\omega} =(\omega_1, \omega_2, \ldots, \omega_m)$.
	
	Consequently, the subgradient norm is bounded, and we have
	\begin{align*}
		\|g^{(k)}\|^2 
		&= \sum_{i=1}^{m} \|g_{a_i}^{(k)}\|^2 + \|g_x^{(k)}\|^2 \\
		&\le m \left[\max_{1 \le j \le m}(\rho_j + \rho_{j+1}) + \max_{1 \le j \le m}\omega_j\right]^2 + \left(\sum_{i=1}^{m} \omega_i\right)^2,
	\end{align*}
	and thus every subgradient of $J$ obeys
	\[
	\|g^{(k)}\| \le G := \sqrt{m \left[\max_{1 \le j \le m}(\rho_j + \rho_{j+1}) + \max_{1 \le j \le m}\omega_j\right]^2 + \left(\sum_{i=1}^{m} \omega_i\right)^2}.
	\]
	
	This explicit constant $G$ plays the same role as the bound $c$ in the classical subgradient analysis~\cite[Proposition 3.2.1, p.~153]{bertsekas2015convex}, ensuring the Projected Subgradient Algorithm convergence properties established below.  
\end{remark}

Here is the main convergence result for the projected subgradient algorithm applied to the Generalized Heron--Waist Problem. Without loss of generality we assume that all weight are positive.


\begin{theorem}\label{thm:psa_convergence}
	Let $C_1, C_2, \ldots, C_m, \mathcal{S}$ be nonempty, closed, and convex subsets of $\mathbb{R}^n$ which are pairwise disjoint, and consider the Generalized Heron--Waist Problem
	\[
	\min_{\mathbf{u} \in \Omega} J(\mathbf{u}) = \sum_{i=1}^{m} \rho_i \|a_i - a_{i+1}\| + \sum_{i=1}^{m} \omega_i \|a_i - x\|,
	\quad 
	\Omega = C_1 \times C_2 \times \cdots \times C_m \times \mathcal{S}.
	\]
	Assume that all weights $\rho_{i}$ and $\omega_i$ are strictly positive. Further assume that at least one of the sets among $\{C_i\}$ and $\mathcal{S}$ is bounded, guaranteeing the existence of an optimal solution $\mathbf{u}^* \in \Omega$ with optimal value $J^* := J(\mathbf{u}^*)$.
	
	Let $\{\mathbf{u}^{(k)}\}$ be the sequence generated by the \emph{Projected Subgradient Algorithm}:
	\begin{equation}
		\label{eq:psa_update}
		\mathbf{u}^{(k+1)} = \proj_{\Omega}\big(\mathbf{u}^{(k)} - \alpha_k g^{(k)}\big),
	\end{equation}
	where $g^{(k)} \in \partial J(\mathbf{u}^{(k)})$ is computed according to \eqref{eq:ga}--\eqref{eq:gx}, $\proj_{\Omega}$ denotes the Euclidean projection onto $\Omega$, and the step-size sequence $\{\alpha_k\}$ satisfies
	\[
	\alpha_k > 0, \qquad 
	\sum_{k=0}^{\infty} \alpha_k = \infty, \qquad 
	\sum_{k=0}^{\infty} \alpha_k^2 < \infty.
	\]
	
	Then the best-iterate sequence of objective values
	\[
	J_{\mathrm{best}}^{(N)} := \min_{0 \le k < N} J(\mathbf{u}^{(k)})
	\]
	converges to the optimal value, i.e., $\lim_{N \to \infty} J_{\mathrm{best}}^{(N)} = J^*$, and $\mathbf{u}^{(k)}$ converges to an optimal solution.
\end{theorem}

\begin{proof}
	The existence of an optimal solution $\mathbf{u}^* \in \Omega$ follows from Theorem~\ref{existencetheorem}. By the nonexpansiveness of Euclidean projections onto closed convex sets~\cite[Proposition 3.2.1, p.~149]{bertsekas2015convex},
	\begin{equation}
		\label{eq:convergence_equation}
		\|\mathbf{u}^{(k+1)} - \mathbf{u}^*\|^2
		= \|\proj_{\Omega}(\mathbf{u}^{(k)} - \alpha_k g^{(k)}) - \proj_{\Omega}(\mathbf{u}^*)\|^2
		\le \|\mathbf{u}^{(k)} - \alpha_k g^{(k)} - \mathbf{u}^*\|^2.
	\end{equation}
	
	Expanding the right-hand side and applying the subgradient inequality $J(\mathbf{u}^{(k)}) - J(\mathbf{u}^*) \le \langle g^{(k)}, \mathbf{u}^{(k)} - \mathbf{u}^* \rangle$ gives
	\begin{align}
		\|\mathbf{u}^{(k)} - \alpha_k g^{(k)} - \mathbf{u}^*\|^2 
		&= \|\mathbf{u}^{(k)} - \mathbf{u}^*\|^2 - 2\alpha_k \langle g^{(k)}, \mathbf{u}^{(k)} - \mathbf{u}^* \rangle + \alpha_k^2 \|g^{(k)}\|^2 \notag \\
		&\le \|\mathbf{u}^{(k)} - \mathbf{u}^*\|^2 - 2\alpha_k [J(\mathbf{u}^{(k)}) - J^*] + \alpha_k^2 \|g^{(k)}\|^2. \label{eq:basic_rec}
	\end{align}
	
	Combining \eqref{eq:convergence_equation} and \eqref{eq:basic_rec} yields
	\begin{equation}
		\label{eq:distance_recursion}
		\|\mathbf{u}^{(k+1)} - \mathbf{u}^*\|^2
		\le \|\mathbf{u}^{(k)} - \mathbf{u}^*\|^2 - 2\alpha_k [J(\mathbf{u}^{(k)}) - J^*] + \alpha_k^2 \|g^{(k)}\|^2.
	\end{equation}
	
	Summing \eqref{eq:distance_recursion} from $k = 0$ to $N-1$ and using the subgradient bound $\|g^{(k)}\| \le G$ from Remark~\ref{rem:subgrad_bound} yields
	\begin{equation}
		\label{eq:sum_bound}
		\sum_{k=0}^{N-1} \alpha_k [J(\mathbf{u}^{(k)}) - J^*]
		\le \frac{1}{2}\|\mathbf{u}^{(0)} - \mathbf{u}^*\|^2 + \frac{G^2}{2}\sum_{k=0}^{N-1} \alpha_k^2.
	\end{equation}
	
	Let $J_{\mathrm{best}}^{(N)} = \min_{0 \le k < N} J(\mathbf{u}^{(k)})$. Since $J(\mathbf{u}^{(k)}) - J^* \ge J_{\mathrm{best}}^{(N)} - J^*$ for all $k < N$, equation \eqref{eq:sum_bound} implies
	\[
	\left(\sum_{k=0}^{N-1} \alpha_k\right) [J_{\mathrm{best}}^{(N)} - J^*]
	\le \sum_{k=0}^{N-1} \alpha_k [J(\mathbf{u}^{(k)}) - J^*]
	\le \frac{1}{2}\|\mathbf{u}^{(0)} - \mathbf{u}^*\|^2 + \frac{G^2}{2}\sum_{k=0}^{N-1} \alpha_k^2.
	\]
	
	Therefore,
	\[
	0 \le J_{\mathrm{best}}^{(N)} - J^*
	\le \frac{\|\mathbf{u}^{(0)} - \mathbf{u}^*\|^2 + G^2\sum_{k=0}^{N-1}\alpha_k^2}{2\sum_{k=0}^{N-1}\alpha_k}.
	\]
	
	Under the step-size assumptions~\cite[Proposition 3.2.6, p.~157]{bertsekas2015convex}, $\sum_{k=0}^{\infty} \alpha_k = \infty$ and $\sum_{k=0}^{\infty} \alpha_k^2 < \infty$, the numerator remains bounded while the denominator diverges. Therefore,
	\[
	\lim_{N \to \infty} [J_{\mathrm{best}}^{(N)} - J^*] = 0,
	\]
	that is, $J_{\mathrm{best}}^{(N)} \to J^*$.
	
	To establish convergence of the iterates $\mathbf{u}^{(k)}$ to an optimal solution, we apply Proposition A.4.4~\cite[p.~462]{bertsekas2015convex} or Proposition A.4.6~\cite[p.~465]{bertsekas2015convex}. In particular, identifying $\beta_k$, $\gamma_k$, and $\delta_k$ in \eqref{eq:distance_recursion}, and setting $\phi(\mathbf{u}^{(k)}, \mathbf{u}^*) = J(\mathbf{u}^{(k)}) - J(\mathbf{u}^*)$, we see that $\mathbf{u}^{(k)}$ converges to some optimal point $\mathbf{u}^*$ satisfying $J(\mathbf{u}^*) = J^*$.
\end{proof}

The complete iterative procedure of the PSA is summarized in Algorithm~\ref{alg:psa_ghwp}.

\begin{algorithm}[H]
	\caption{Projected Subgradient Algorithm (PSA) for the Generalized Heron--Waist Problem}
	\label{alg:psa_ghwp}
	
	\KwIn{
		Nonempty, closed, convex sets $C_1,\ldots,C_m,\mathcal S \subset \mathbb{R}^n$;\\
		Nonnegative weights $\boldsymbol{\rho}=(\rho_1,\ldots,\rho_m)$, 
		$\boldsymbol{\omega}=(\omega_1,\ldots,\omega_m)$;\\
		Initial feasible point $\mathbf{u}^{(0)}=(a_1^{(0)},\ldots,a_m^{(0)},x^{(0)})\in\Omega$;\\
		Tolerance $\varepsilon>0$, maximum iterations $K_{\max}$.
	}
	
	\KwOut{
		Approximate optimal solution $\mathbf{u}^*$ and objective value $J^*$.
	}
	
	\BlankLine
	\textbf{Initialization:}  
	Set $k\leftarrow1$,  
	$\mathbf{u}_{\mathrm{best}}\leftarrow\mathbf{u}^{(0)}$,  
	$J_{\mathrm{best}}\leftarrow J(\mathbf{u}^{(0)})$\;
	
	\BlankLine
	\For{$k=1,2,\ldots,K_{\max}$}{
		
		\tcp{Subgradient computation}
		Compute $g_{a_i}^{(k)}$ using \eqref{eq:ga} for $i=1,\ldots,m$\;
		Compute $g_x^{(k)}$ using $\eqref{eq:gx}$\;
		
		$g^{(k)}\leftarrow(g_{a_1}^{(k)},\ldots,g_{a_m}^{(k)},g_x^{(k)})$\;
		
		\tcp{Step size}
		Set $\alpha_k\leftarrow\dfrac{1}{k}$\;
		
		\tcp{Subgradient step}
		$\tilde{\mathbf{u}}^{(k+1)}\leftarrow \mathbf{u}^{(k)}-\alpha_k g^{(k)}$\;
		
		\tcp{Projection onto the feasible set $\Omega$}
		$\mathbf{u}^{(k+1)}\leftarrow \proj_{\Omega}\!\bigl(\tilde{\mathbf{u}}^{(k+1)}\bigr)$\;
		
		\tcp{Objective update}
		$J^{(k+1)}\leftarrow J(\mathbf{u}^{(k+1)})$\;
		
		\If{$J^{(k+1)}<J_{\mathrm{best}}$}{
			$J_{\mathrm{best}}\leftarrow J^{(k+1)}$\;
			$\mathbf{u}_{\mathrm{best}}\leftarrow\mathbf{u}^{(k+1)}$\;
		}
		
		\If{$\|g^{(k)}\|<\varepsilon$ \textbf{or}
			$|J^{(k+1)}-J^{(k)}|<\varepsilon$}{
			\textbf{stop and return} $\mathbf{u}_{\mathrm{best}},J_{\mathrm{best}}$\;
		}
	}
\end{algorithm}

\section{Numerical Illustrations}
\label{numerical}

In this section, we validate the theoretical findings and convergence properties of the proposed algorithm through numerical experiments.
We present numerical illustrations that demonstrate the practical performance of the
Projected Subgradient Algorithm (PSA).
The examples are designed to highlight the behavior of the algorithm under different
weighting schemes and complex convex geometries, as well as to verify the existence,
uniqueness, and optimality properties established earlier.
We consider two representative examples: a planar configuration in
$\mathbb{R}^2$ and a spatial configuration in $\mathbb{R}^3$.
Together, these examples illustrate the scalability of the method and its robustness
across different dimensions.

\paragraph{Experimental Setup:}

All experiments were implemented in Python~3.11 on Google~Colab using an
Intel\textsuperscript{\textregistered} Xeon\textsuperscript{\textregistered}
CPU @~2.20\,GHz with 12.7\,GB of RAM.
The Projected Subgradient Algorithm described in Algorithm~\ref{alg:psa_ghwp} was employed with the diminishing step-size rule $\alpha_k = 1/k$, which satisfies the classical convergence conditions for
nonsmooth convex optimization:
\begin{equation}
	\alpha_k > 0, \quad \lim_{k\to\infty} \alpha_k = 0, \quad
	\sum_{k=1}^{\infty} \alpha_k = \infty, \quad \text{and} \quad
	\sum_{k=1}^{\infty} \alpha_k^2 < \infty.
\end{equation}
The stopping criterion was defined based on stagnation of the objective value.
Specifically, the algorithm terminates when $|J(u^{(k+1)}) - J(u^{(k)})| < 10^{-12},$ where $J(u)$ denotes the objective function defined in~\eqref{eq:objective_ghwp}.

\begin{example}\label{GHWPEX1}
We consider a two-dimensional instance of the Generalized Heron--Waist Problem
involving four chain constraint sets and a single hub constraint set.
The chain sets $C_1, \ldots, C_4 \subset \mathbb{R}^2$ are chosen as closed Euclidean
discs of unit radius, centered at
$c_1 = (7,-6), c_2 = (4,5), c_3 = (-3,4), c_4 = (-6,-4).$ The hub constraint set $\mathcal{S}$ is defined as a closed square centered at $(1,-1)$ with half-side length equal to $1$.
To examine the influence of different weights in the objective function, we
assign nonuniform weights to both the cyclic perimeter terms and the radial terms.
Specifically, the waist (chain) weights and the Heron (radial) weights are given by
$\boldsymbol{\rho} = (1,\,2,\,2,\,1)$  and $\boldsymbol{\omega} = (2,\,1,\,1,\,2),$ respectively.
\medskip
The Projected Subgradient Algorithm was initialized from the feasible configuration
$a_1^{(0)}=(8,-6)$, $a_2^{(0)}=(4,6)$, $a_3^{(0)}=(-3,5)$, $a_4^{(0)}=(-7,-4)$,
and $x^{(0)}=(2,-2)$.
The algorithm was then executed using the diminishing step-size rule
$\alpha_k = 1/k$ and a stopping tolerance of $10^{-12}$. The Projected Subgradient Algorithm converged to a high-precision optimal configuration
under the prescribed tolerance.
The optimal locations of the chain variables $\mathbf{a}_i^*$ together with the optimal
hub position $\mathbf{x}^*$ are reported in Table~\ref{tab:optimal_solution_ex1}.
The convergence behavior of the algorithm, including the iteration counts and the corresponding CPU times required to attain successive accuracy levels, is summarized in Table~\ref{tab:accuracy_cpu_ex1}.
For the tolerance level $10^{-12}$, the algorithm achieved convergence after $120{,}691$ iterations with a total computational time of $36.87$ seconds, yielding a final objective value of $J^* = 85.2772730788.$ Figure~\ref{fig:example1} illustrates the final optimal configuration, clearly depicting
the interaction between the convex constraint sets, the resulting optimal polygonal
chain, and the hub connections induced by the different weight.

\end{example}

\begin{figure}[htbp]
	\centering
	\includegraphics[height=8cm]{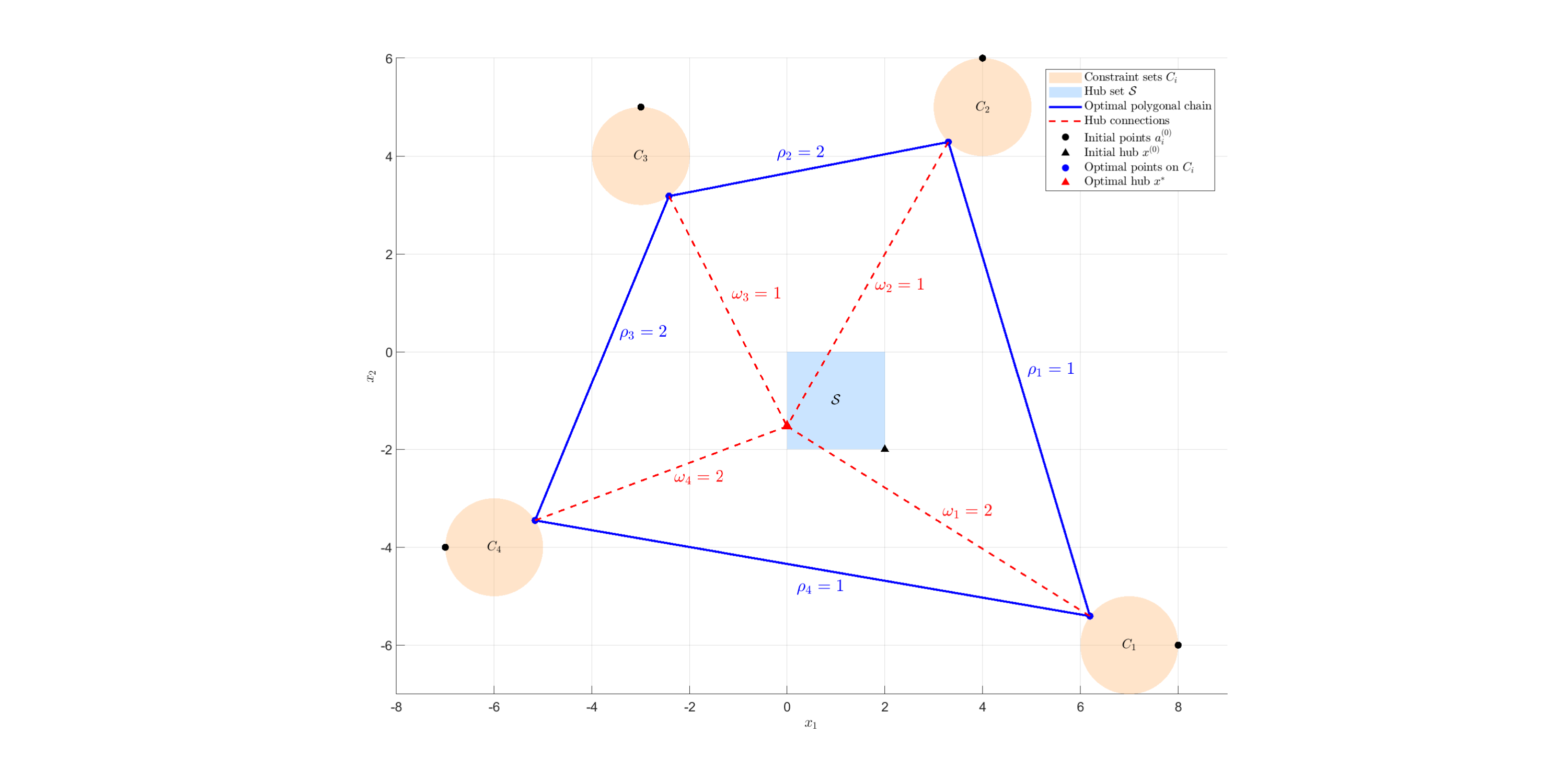}
	\caption{Final optimal configuration for Example~\ref{GHWPEX1} in $\mathbb{R}^2$.}
	\label{fig:example1}
\end{figure}

\begin{table}[H]
	\centering

	\begin{tabular}{ccc}
		\toprule
		\textbf{Variable} & \textbf{$x_1$-coordinate} & \textbf{$x_2$-coordinate} \\
		\midrule
		$a_1^*$ & 6.1952593003 & $-5.4063735128$ \\
		$a_2^*$ & 3.2987886747 & 4.2870465112 \\
		$a_3^*$ & $-2.4248918253$ & 3.1819226275 \\
		$a_4^*$ & $-5.1638552312$ & $-3.4514911799$ \\
		\midrule
		$x^*$   & 0.0000000000 & $-1.5290885465$ \\
		\bottomrule
	\end{tabular}
	\caption{Optimal chain points and hub location for Example~\ref{GHWPEX1}}
		\label{tab:optimal_solution_ex1}
\end{table}

\begin{table}[H]
	\centering
	\begin{tabular}{ccc}
		\toprule
		\textbf{Tolerance} & \textbf{Iterations} & \textbf{CPU Time (seconds)} \\
		\midrule
		$10^{-4}$  & 96       & 0.0416 \\
		$10^{-6}$  & 169      & 0.0720 \\
		$10^{-8}$  & 1{,}502  & 0.7497 \\
		$10^{-10}$ & 13{,}544 & 8.5284 \\
		$10^{-12}$ & 120{,}691 & 36.8669 \\
		\bottomrule
	\end{tabular}
	\caption{Convergence accuracy versus iteration count and CPU time for Example~\ref{GHWPEX1}.}
	\label{tab:accuracy_cpu_ex1}
\end{table}

\begin{example}\label{GHWPEX2}
	We next consider a three-dimensional instance of the Generalized Heron--Waist
	Problem in order to assess the performance of the proposed Projected
	Subgradient Algorithm in higher dimensions and under more intricate geometric
	constraints. The chain constraint sets $C_1,\ldots,C_5 \subset \mathbb{R}^3$ are modeled as
	axis-aligned cubes with half side length equal to $1$, whose centers are given by
	$(-2,-4,2)$, $(5,6,5)$, $(1,7,3)$, $(-4,2,-3)$, and $(6,-6,-2)$, respectively.
	Each cube therefore defines a compact and convex feasible region for the
	corresponding chain variable, introducing multiple boundary-induced geometric
	effects into the optimization process. 	The hub constraint set $\mathcal{S}$ is taken as a closed Euclidean sphere of radius $1$ centered at $(-1,2,1)$. In contrast to the polyhedral chain sets, this
	spherical hub introduces a smooth radial constraint that couples all chain
	variables through weighted distance interactions, thereby enriching the overall
	problem structure.

	To capture asymmetric interaction strengths, distinct weights are assigned to
	the cyclic chain terms and the hub connections. In particular, the waist (chain)
	weights and Heron (radial) weights are chosen as
	\[
\boldsymbol{\rho} = (1.5,\,1.1,\,1.2,\,0.9,\,0.8),
\qquad
\boldsymbol{\omega} = (1.0,\,1.3,\,1.5,\,1.0,\,0.95),
\] 
respectively. This choice induces asymmetric geometric forces along both the polygonal chain and the hub
	connections, allowing us to observe how the algorithm adapts to nonuniform
	weighting patterns.
	
	The Projected Subgradient Algorithm was initialized from the feasible
	configuration $a_1^{(0)}=(-1,-5,1)$, $a_2^{(0)}=(4,7,6)$, $a_3^{(0)}=(0,8,2)$,
	$a_4^{(0)}=(-5,3,-4)$, $a_5^{(0)}=(7,-7,-3)$, and $x^{(0)}=(-1,3,1)$, with each
	initial point lying within its corresponding constraint set. The algorithm was
	then executed using the diminishing step-size rule $\alpha_k = 1/k$ and a stopping
	tolerance of $10^{-12}$.

	The computed optimal chain points and hub location are reported in
	Table~\ref{tab:optimal_solution_ex2}, while the convergence behavior of the
	algorithm, including iteration counts and the corresponding CPU times required to
	achieve successive accuracy levels, is summarized in
	Table~\ref{tab:accuracy_cpu_ex2}. Figure~\ref{fig:example2} illustrates the final optimal configuration in $\mathbb{R}^3$,
	showing the constraint sets, the optimal polygonal chain, and the weighted hub
	connections. This example demonstrates the robustness and scalability of the proposed
	method for higher-dimensional and geometrically constrained settings of the
	Generalized Heron--Waist Problem.
	
\end{example}

\begin{table}[H]
	\centering
	\begin{tabular}{cccc}
		\toprule
		\textbf{Variable} & \textbf{$x_1$-coordinate} & \textbf{$x_2$-coordinate} & \textbf{$x_3$-coordinate} \\
		\midrule
		$a_1^*$ & $-1.0000000000$ & $-3.0000000000$ & $ 1.5531128057$ \\
		$a_2^*$ & $ 4.0000000000$ & $ 5.0000000000$ & $ 4.0000000000$ \\
		$a_3^*$ & $ 0.5587959267$ & $ 6.0000000000$ & $ 2.0000000000$ \\
		$a_4^*$ & $-3.0000000000$ & $ 2.5330955112$ & $-2.0000000000$ \\
		$a_5^*$ & $ 5.0000000000$ & $-5.0000000000$ & $-1.0000000000$ \\
		\midrule
		$x^*$   & $-0.0785608029$ & $ 2.3641633583$ & $ 1.1354062574$ \\
		\bottomrule
	\end{tabular}
	\caption{Optimal chain points and hub location for Example~\ref{GHWPEX2} in $\mathbb{R}^3$.}
	\label{tab:optimal_solution_ex2}
\end{table}

\begin{table}[H]
	\centering
	\begin{tabular}{ccc}
		\toprule
		\textbf{Tolerance} & \textbf{Iterations} & \textbf{CPU Time (seconds)} \\
		\midrule
		$10^{-4}$  & 77        & 0.1064 \\
		$10^{-6}$  & 1{,}059   & 0.6394 \\
		$10^{-8}$  & 15{,}343  & 6.0517 \\
		$10^{-10}$ & 228{,}379 & 71.9071 \\
		$10^{-12}$ & 3{,}389{,}834 & 1{,}053.2593 \\
		\bottomrule
	\end{tabular}
	\caption{Convergence accuracy versus iteration count and CPU time for Example~\ref{GHWPEX2} in $\mathbb{R}^3$ using the Projected Subgradient Algorithm.}
	\label{tab:accuracy_cpu_ex2}
\end{table}
\begin{figure}[H]
	\centering
	\includegraphics[width=0.99\textwidth]{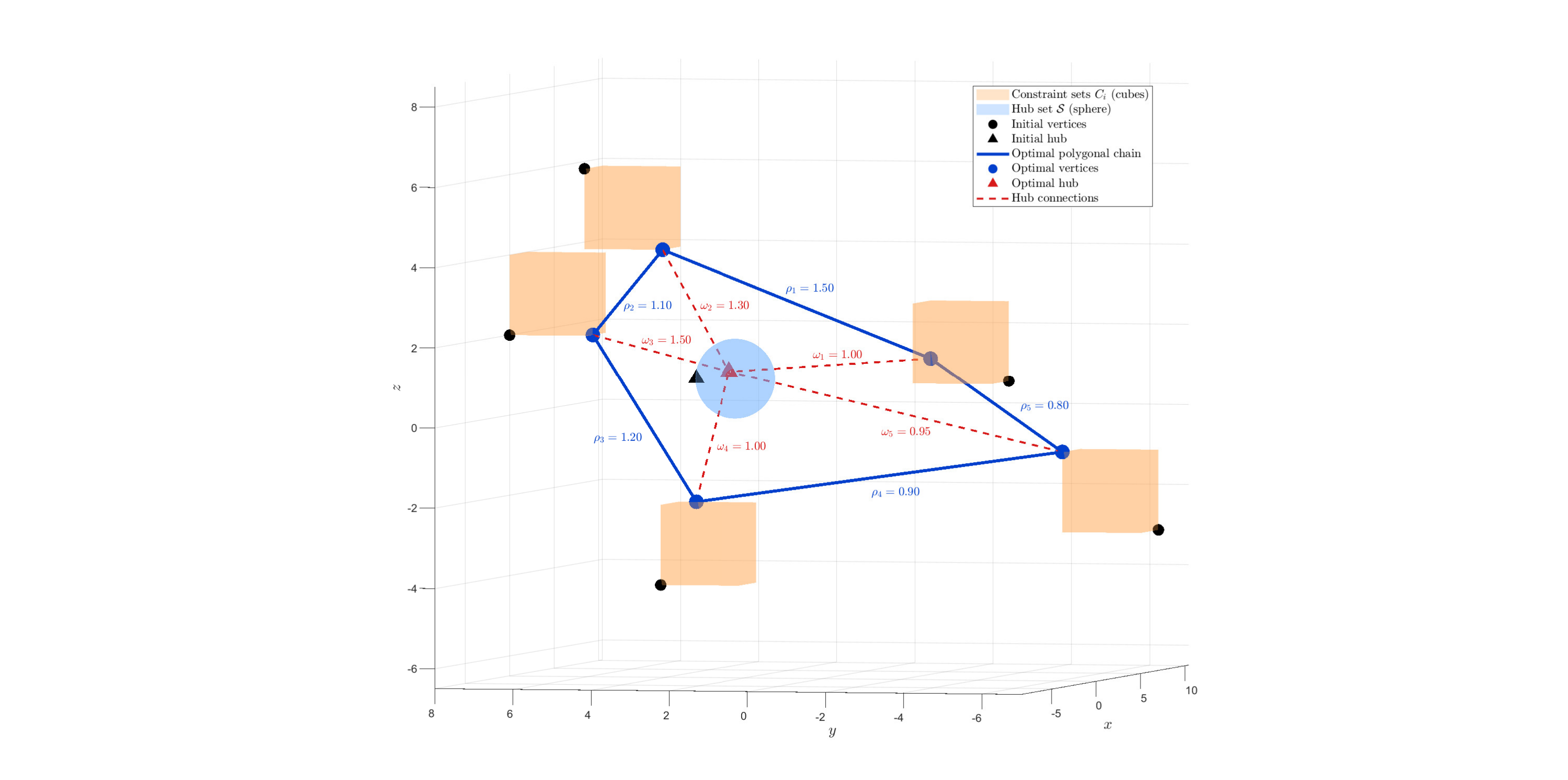}
	\caption{Final optimal configuration for Example~\ref{GHWPEX2} in $\mathbb{R}^3$.}
	\label{fig:example2}
\end{figure}

\section{Conclusion}\label{conclusion}

This paper introduces the Generalized Heron-Waist Problem, extending the classical Heron hub-location and waist perimeter-minimization problems within a single convex optimization framework. By extending point constraints to arbitrary convex sets and incorporating heterogeneous weights, we have created a flexible model for hybrid network systems requiring simultaneous optimization of cyclic connectivity and radial hub access---structures fundamental to supply chains, transportation networks, and communication infrastructures.

We establish complete theoretical foundations through convex analysis: proving existence under natural boundedness assumptions, demonstrating uniqueness when sets are strictly convex with positive weights, and deriving necessary and sufficient optimality conditions as elegant force-balance equilibria. These conditions extend classical geometric principles into modern subdifferential calculus, revealing how weighted directional forces and constraint normals equilibrate at optimal configurations.

Our Projected Subgradient Algorithm employs the problem's convex structure to guarantee convergence under standard step-size rules. Numerical experiments in $\mathbb{R}^2$ and $\mathbb{R}^3$ validate practical effectiveness, achieving high-precision solutions across diverse geometries and weighting schemes within reasonable computational time.

By bridging hub-centric and perimeter-centric optimization paradigms, the GHWP provides both theoretical insights and computational tools for modern network design challenges. Future directions include acceleration techniques, stochastic extensions, and applications to concrete engineering problems, demonstrating how classical geometric optimization continues to illuminate contemporary challenges in facility location and network infrastructure.

\section*{Author Contributions}
All authors contributed equally to this work.

\section*{Conflict of Interest}
The authors declare that there is no conflict of interest.

\section*{Funding}
This research received no external funding.

\bibliographystyle{plain}  
\bibliography{Heron_Waist_Ref}

\end{document}